\documentclass[12pt]{article}
\usepackage{amssymb}
\usepackage{amsmath}
\usepackage{mathtools}
\usepackage[utf8]{inputenc}
\usepackage[breaklinks=true]{hyperref}
\usepackage{comment}
\usepackage{mathrsfs}
\usepackage{geometry}
\usepackage{amsthm}
\usepackage{amsrefs}
\usepackage{float}
\usepackage[T1]{fontenc}
\usepackage{graphicx} 

\numberwithin{equation}{section}

\newtheorem{theorem}{Theorem}[section]

\newtheorem{lemma}[theorem]{Lemma}

\newtheorem{remark}[theorem]{Remark}

\DeclareMathOperator*{\argmax}{arg\,max}
\newcommand{\E}{\mathbb{E}}

\newcommand{\N}{\mathbb{N}}
\newcommand{\R}{\mathbb{R}}
\newcommand*{\boldone}{\text{\usefont{U}{bbold}{m}{n}1}}

\title{Asymptotic behavior of the extremal position in a multi-type branching random walk with heavy-tailed displacements}
\author{Krzysztof Kowalski\footnote{University of Wrocław, Poland. Email: krzysztof.kowalski@math.uni.wroc.pl. Supported by the National Science Center, Poland (OPUS, grants 2019/33/B/ST1/00207 and 2020/39/B/ST1/00209)
\\ MSC2020 subject classification: 60J80, 60F05, 60F15
\\ keywords: branching random walk, maximal position, central limit theorem, heavy-tailed distributions}}

\begin{document}
\maketitle
\begin{abstract}
    We consider a multi-type branching random walk with displacements that have either regularly varying or semi-exponential tails. We investigate the behavior of the rightmost particle in irreducible and reducible regimes and identify the correct normalization for different settings. 
\end{abstract}

\tableofcontents

\section{Introduction}

\subsection{Branching random walks}Formally, a branching random walk, abbreviated as BRW, is constructed as follows. The process starts with a single particle placed at the origin. Given a point process ${\mathcal Z} = \sum_{k=1}^N \delta_{\xi_k}$ on $\R$, where $N$, denoting the size of the offspring, is a random variable on $\mathbb{N}_0$, the original particle at time 1 dies and gives birth to $N$ particles positioned according to ${\mathcal Z}$. 
These particles are called the first generation of the process. At time 2, each of these particles reproduces independently
and has offspring with positions relative to their parents' position given by an independent copy of ${\mathcal Z}$. 

The process continues infinitely. As a result, we obtain a marked tree $\left(S, \mathbb{T}\right)$, where the tree $\mathbb{T}$ is the set of all particles equipped with the natural tree structure, and $S_v$ is the position of a given particle $v \in \mathbb{T}$.

We write $|v|$ for the generation of $v$ and $m=\mathbb{E}[N]$ for the mean number of offspring. For a BRW with displacements given by $\xi$,
let \[R_n = \sup_{|v|=n} S_v\] denote the position of the most right particle at time $n$. The asymptotic behavior of $R_n$ is most commonly studied under the following exponential moment assumption:
\begin{align}\label{ExponentialMoment}
    \text{there exists } \theta >0, \text{ such that } \mathbb{E}\left[ \sum_{i=1}^n e^{\theta \xi_i} \right] < \infty.
\end{align} 

Under \eqref{ExponentialMoment}, Biggins \cite{Biggins1976} proved in 1976 the law of large numbers for $R_n$,  i.e.  $\frac {R_n} n$ converges almost surely to a constant $\alpha$.

The corresponding second order limit theorem was proved by Aïdékon \cite{Aidekon2013} in 2013, who showed that for some $c>0$, $R_n - \alpha  n + c \log n $ converges in distribution to a random shift of the Gumbel distribution.

The assumption \eqref{ExponentialMoment} is critical to the linear growth of $R_n$. Durrett \cite{Durrett1983} showed in 1983 that if one assumes instead that the displacements have regularly varying tails, $R_n$ grows exponentially fast. More specifically, assume that for some slowly varying $L$ and some $r>0$, we have
\begin{align}\label{tailsdurrett}
    \mathbb{P}\left(\xi > x\right) \sim L(x) x^{-r} \hspace{0.5cm} \text{as }x \rightarrow \infty
\end{align}
and
\begin{align}\label{tailsdurrett2}
     \log(-x)\mathbb{P}\left(\xi \leq x\right)\rightarrow 0 , \hspace{0.5cm}  \text{as }x \rightarrow -\infty.
\end{align}
Then
\[
\mathbb{P} ( R_n \leq a_n x)  \xrightarrow[ n \rightarrow \infty]{} \E[e^{-c Wx^{-r}}]
\]
where $c>0$ is a constant, $W$ is a random variable depending on the underlying Galton-Watson process, and $\{a_n\}_{n \in \N}$ satisfies 
\begin{equation}\label{an}
    m^n \mathbb{P}\left(\xi > a_n\right) \xrightarrow[ n \rightarrow \infty]{} 1.
\end{equation}

Another model present in the literature considers displacements with semi-exponential tails. Assume that for some slowly varying $a$, $L$ and $r \in (0,1)$, 
\begin{equation*}
    \mathbb{P}\left(\xi > x\right) = a(x) \exp\{-L(x)x^{r}\}.
\end{equation*}
Then, according to Gantert \cite{Gantert2000}, 

\begin{equation*}
 \frac {R_n} {b_n} \xrightarrow[ n \rightarrow \infty]{a.s.} (\log m)^{\frac 1 r},   
\end{equation*}
where $b_n$ satisfies
\begin{equation}\label{eq:gantert}
\frac {L(b_n)b_n^{r}} n \xrightarrow[ n \rightarrow \infty]{} 1.
\end{equation}
This model was further explored in a series of papers by Dyszewski, Gantert, and Höfelsauer in the context of large deviations \cite{Dyszewski2020}, extremal point process \cite{Dyszewski2022} and second-order fluctuations \cite{Dyszewski2023}.

\subsection{Multi-type branching random walks}
 Multi-type branching random walks extend the ideas of one-type branching processes to a multidimensional setting, which is necessary to model various phenomena, such as cell population dynamics with different phenotypes. Formally, a multi-type branching random walk is constructed analogously to the one-type model. Take a set of types $\mathcal{C} = \{1, 2, \ldots, d\}$ and a corresponding family of point processes $\{\mathcal{Z}_{ij}\}_{i,j \in \mathcal{C}}$, where $\mathcal{Z}_{ij} = \sum_{k=1}^{N_{i,j}}\delta_{\xi^{j}_k}$, and for each $j \in \mathcal{C}$, $\{\xi_k^j\}_{k \in \mathbb{N}}$ are marginally identically distributed. We start with a single particle of any given type $i$ placed at the origin. For each $j \in \mathcal{C}$, this particle gives birth to $N_{i,j}$ children of type $j$, positioned according to $\mathcal{Z}_{ij}$, and subsequently dies. At time 2, each particle of type $j$ reproduces independently according to copies of $\{\mathcal{Z}_{jk}\}_{k \in \mathcal{C}}$, and subsequently dies. The process continues infinitely.

In this case, the number of offspring depends on the type of parent, but the displacement of a particle depends only on its own type. We write $Z_n = \{Z^1_n, Z^2_n, \dots, Z^d_n\}$ for the d-dimensional Galton-Watson process recording the number of particles of each type in the $n$-th generation and define the mean matrix $M = (\mathbb{E}[N_{i,j}])_{i,j \in \mathcal{C}}$. Since all entries in $M$ are nonnegative, it has the principal (although possibly not unique) eigenvalue that we denote by $\rho$. We utilize the one-type notation and define $\sigma(v) = i$ whenever $v$ belongs to type $i$. Our main point of interest is again the asymptotic behavior of the maximum position $R_n$. 

In the multi-type model, one needs to distinguish between two significantly different regimes. We call the process \textbf{irreducible} if a particle of any given type can appear in any line of descent with positive probability and \textbf{reducible} otherwise. In terms of the mean matrix $M$, irreducibility translates to the following statement: for any $i,j \in \mathcal{C}$, there exists $n \in \mathbb{N}$, such that $M^n(i,j) > 0$. The previous results on the multi-type model under the exponential moment assumption go back to Biggins \cite{Biggins1976}, who showed in 1976 that the irreducible model exhibits linear growth and described the limiting constant. The reducible case proved to be more challenging. Weinberger et al.\cite{Weinberger2002} in 2002 argued that the spreading speed should be the maximum of speeds of the types considered separately, essentially ignoring the interaction between types. A flaw in this argument was identified by Weinberger et al. \cite{Weinberger2007} in 2007, and the correct limiting constant was ultimately described by Biggins \cite{Biggins2012} in 2012. As it turns out, the interplay between the types can significantly increase the growth speed. This effect was called \textbf{ anomalous spreading} in Weinberger et al. \cite{Weinberger2007} and makes the study of reducible models particularly appealing.

Bhattacharya, Maulik, Palmowski and Roy \cite{Palmowski} in 2019 considered an irreducible model with displacements having regularly varying tails. They showed the convergence of the extremal process to a randomly scaled scale-decorated Poisson point process and, as a result, obtained a limit theorem for the maximum position. In this case, it turns out that the behavior is analogous to the one-type model considered in \cite{Durrett1983}, with the largest eigenvalue of $M$ replacing the mean number of offspring and the heaviest tail dominating the lighter ones. 

In this paper, we provide a corresponding result for the reducible case and describe the complete asymptotics of the extremal position in the previously unstudied multi-type model with semi-exponentially tailed displacements.

\section{Irreducible multi-type branching random walk}

In this chapter, we present the results on the irreducible branching random walks. We adopt the notation introduced in the previous chapter and denote \[\mathbb{P}_i(\cdot)=\mathbb{P}\left(\cdot \ | \ \text{initial particle is of type } i \right), \]
and $\mathbb{E}_i$ for the expectation with respect to $\mathbb{P}_i$. Whenever the index is omitted, we assume that the initial particle is of type 1.
Throughout this chapter, we make the following assumptions on the underlying Galton-Watson process.
Firstly, there exists $l \in \mathbb{N}_+$ such that 
\begin{equation}\label{as3}
    M^l(i,j) > 0 \text{ for all } i,j \in \mathcal{C}.
\end{equation} 
This assumption guarantees irreducibility, and through the Perron-Frobenius theorem, it asserts that $\rho$, the principal eigenvalue of $M$, is simple. We also assume \begin{equation}\label{as1} 
    \rho > 1,
\end{equation} 
ensuring that the process survives with positive probability (see \cite{harris}, Theorem 7.1). Finally, we assume that the Kesten-Stigum condition,
\begin{equation}\label{as2}
    \mathbb{E}_i[Z_1^j \log Z_1^j] < \infty,
\end{equation}
holds for all $i,j \in \mathcal{C}$.
Under these assumptions, the well known Kesten and Stigum theorem \cite{kesten} asserts that for any $i \in \mathcal{C}$, \begin{equation}\label{limit}
    \frac{Z_n} {\rho^n}  \rightarrow Wu \hspace{0.5cm} \mathbb{P}_i\text{ -a.s.},
\end{equation} where $W$ is a non-degenerate random variable and $u$ is the left eigenvector of $M$. It is a straightforward conclusion that 
\begin{equation}\label{limitscalar}
\frac{Z_n} {\rho^n} \cdot v \rightarrow W \hspace{0.5cm} \mathbb{P}_i\text{ -a.s.}
\end{equation}If we write $v$ for the right eigenvector of $M$, normalized so that $u\cdot v = 1$, then we also have
\begin{equation}\label{heavytailmean}
        \mathbb{E}_i[W] = v_i.
\end{equation}
 To avoid conditioning on the survival set, we assume $  \mathbb{P}(  Z_n \rightarrow 0) = 0$. 

\subsection{Displacements with regularly varying tails}

Let $F_i(x) = \mathbb{P}\left(\xi^i \leq x\right)$. In this section, we assume that the displacements are independent and there exist slowly varying functions $\{L_i\}_{i \in \mathcal{C}}$ and positive constants $\{r_i\}_{i \in \mathcal{C}}$, satisfying
\begin{align}\label{tails1}
\begin{aligned}
     &1-F_i(x) \sim L_i(x) x^{-r_i} \hspace{0.5cm} \text{as }x \rightarrow \infty, \\
     &\log(-x)F_i(x)\rightarrow 0 , \hspace{0.5cm}  \text{as }x \rightarrow -\infty.
\end{aligned}
\end{align}
 These assumptions are a natural extension of the one-type case considered in \cite{Durrett1983}. For simplicity, we additionally assume the existence of a unique heaviest tail: There exists $I \in \mathcal{C}$, satisfying $r_I < r_j$ for all $j\neq I$. To simplify the notation, we write $r=r_I$. Our main result is the following theorem.

\begin{theorem}\label{t1} Let \[\zeta =u_I\sum_{j>0} \rho^{-j} \sum_{l \in \mathcal{C}}    \mathbb{P}_I\left(Z_j^l> 0\right).\] and choose the sequence $\{a_n\}_{n \in \N}$ so that 
\begin{equation}\label{t1:an}
    \rho^n \left(1-F_I(a_n)\right) \xrightarrow[ n \rightarrow \infty]{} 1.
\end{equation}
Then
\[
\mathbb{P} ( R_n \leq a_n x)  \xrightarrow[ n \rightarrow \infty]{} \E[e^{-\zeta Wx^{-q}}]
\]
\end{theorem}

\begin{remark}\label{t1:remark}\emph{
    As in the one-type case, the existence of $a_n$ satisfying \eqref{t1:an} is guaranteed by the result of de Bruijn \cite{DeBruijn}. If $L^\#$ is the de Bruijn conjugate of $L$, one can take $a_n = L^\#\left(\rho^{\frac n {r}}\right) \rho^{\frac n {r}}$. In particular, this guarantees that for any $\varepsilon > 0$, \begin{equation}
        \rho^{\frac n r (1-\varepsilon)} < a_n < \rho^{\frac n r (1+\varepsilon)}
    \end{equation}
    for sufficiently large $n$.
}\end{remark}

 \begin{remark} \emph{The result partly overlaps with Corollary 3.4 from \cite{Palmowski}, however there are several differences. We allow the existence of leaves in our tree, and we present a direct argument, in contrast to the result being a conclusion from the convergence of the extremal process. On the other hand, we assume independence of the displacements, as opposed to the more general notion of point processes converging in suitable topology.
}\end{remark}

\begin{proof}[Proof of Theorem \ref{t1}]

To begin, we present a lemma that characterizes the asymptotic behavior of the total population in an irreducible multi-type Galton-Watson process. For $i \in \mathcal{C}$, let \[
    Y^i_n = \left|\bigcup_{k =1}^{n}\left\{ v \in \mathbb{T}_k \ : \ \sigma(v) =i, (\exists w \in \mathbb{T}_n)(w_k = v) \right\}\right|
    \]
be the total number of particles of type $i$ that have offspring in the $n$-th generation.
\begin{lemma}\label{totalpopulation}
    Let  \[\zeta_i =u_i\sum_{j>0} \rho^{-j} \sum_{l \in \mathcal{C}}    \mathbb{P}_i\left(Z_j^l> 0\right).\] Then for all $i \in \mathcal{C}$,
    \[
    \frac{Y^i_n} {\rho^n}     \xrightarrow[ n \rightarrow \infty]{a.s.} \zeta_i W.
    \]

\end{lemma}

\begin{proof}
Observe
\[
Y^i_n = \sum_{j=0}^{n-1}\sum_{l \in \mathcal{C}} \sum_{k=1}^{Z^i_{n-j} }\boldone_{\{Z^l_j(i,k) > 0\}}
\]
where for any $j$ and $l$, $\{Z^l_j{(i,k)}\}_{k>0} $ are i.i.d. distributed as $Z^l_j$ under $\mathbb{P}_i$, and for $i_1\neq i_2$, $\{Z^l_j{(i_1,k)}\}_{k>0} $ are independent of $\{Z^l_j{(i_2,k)}\}_{k>0} $. Hence
\begin{align*}
    \frac {Y^i_n} {\rho^n} = \sum_{j=0}^{n-1} \rho^{-j}  \sum_{l \in \mathcal{C}}\frac {Z^i_{n-j}} {\rho^{n-j}}  \frac 1 {Z^i_{n-j}} \sum_{k=1}^{Z_{n-j}^i }\boldone_{\{Z^l_j{(i,k)}>0 \}}
\end{align*}
Denote $D_{n-j} = \frac {Z^i_{n-j}} {\rho^{n-j}} $ and $E^l_{n-j} = \frac 1 {Z^i_{n-j}} \sum_{k=1}^{Z_{n-j}^i }\boldone_{\{Z^l_j{(i,k)} > 0\}}$. For any fixed $j >0$, by the strong law of large numbers, $E^l_{n-j}  \xrightarrow[ n \rightarrow \infty]{a.s.} \mathbb{P}_i(Z^l_j > 0) $, and by \eqref{limit},  $D_{n-j}  \xrightarrow[ n \rightarrow \infty]{a.s.}Wu_i$.

Now fix $N>0$. Then for $n>N$
\begin{align*}
     \frac {Y^i_n} {\rho^n} &\leq \sum_{j=0}^{N} \rho^{-j}  \sum_{l \in \mathcal{C}} D_{n-j} E^l_{n-j} +  \sum_{l \in \mathcal{C}}\sup_{k\geq N+1}\{D_k \}\sum_{j=N+1}^\infty \rho^{-j}  \\&=  \sum_{j=0}^{N} \rho^{-j}  \sum_{l \in \mathcal{C}} D_{n-j} E^l_{n-j} +  d\frac {\rho^{-N}} {\rho - 1} \sup_{k\geq N+1}\{D_k \}
\end{align*}
So $\mathbb{P}$-almost surely
\[
\limsup_{n} \frac {Y^i_n} {\rho^n} \leq W u_i\sum_{j=0}^{N}  \rho^{-j}  \sum_{l \in \mathcal{C}} \mathbb{P}_i(Z^l_j > 0) + d\frac {\rho^{-N}} {\rho - 1} \sup_{k\geq N+1}\{D_k \}
\]
Letting $N \rightarrow \infty$ we get 
\[
\limsup_{n} \frac {Y^i_n} {\rho^n} \leq \zeta_iW \hspace{0.5cm} \mathbb{P}-\text{a.s.}
\]
For the bound from below, we note that
\[
    \frac {Y^i_n} {\rho^n} \geq \sum_{j=0}^{N}\rho^{-j}  \sum_{l \in \mathcal{C}} D_{n-j} E^l_{n-j} 
\]
so taking $\liminf\limits_n$ and then letting $N \rightarrow \infty$ we get
\[
\liminf_{n} \frac {Y^i_n} {\rho^n} \geq \zeta_iW \hspace{0.5cm} \text{a.s.}
\]
which concludes the proof of the lemma.
\end{proof}

Now let 
\[
M_n = \max\{\xi_{v_k} \ : \ v \in \mathbb{T}_n, k\leq n\}
\]Since up to the $n$-th generation, for any $i \in \mathcal{C}$, there are $Y^i_n$ displacements with distribution $F_i$, and all the displacements are independent of each other, we have
\[
\mathbb{P}\left(M_n \leq a_n x\right) = \mathbb{E} \left[\prod_{i \in \mathcal{C}} F_i(a_nx)^{Y^i_n}\right]
\]
Note that 
\begin{align}\label{t1:eq2}
    F_i(a_nx)^{Y^i_n} = \exp{\Big\{\frac {Y^i_n} {\rho^n} \rho^n \log F_i(a_n x)\Big\}} 
\end{align}
 Now, $a_n$ was chosen so that $\rho^n (1- F_I(a_n )) \xrightarrow[ n \rightarrow \infty]{}  1$. Furthermore, $1 - F_i(z) \sim L_i(z) z^{-r_i} $ as $z \rightarrow \infty$, where $L_i$ is slowly varying. Then, using the fact that for $z$ close to $0$, $\log(1+z) \sim z$, we have
\begin{align}\label{t1:eq4}
\rho^n \log F_I(a_n x) \sim -\rho^n (1- F_I(a_n x )) = -\rho^n (1- F_I(a_n) ) \frac { (1- F_I(a_n x )) } {(1- F_I(a_n ))} \xrightarrow[ n \rightarrow \infty]{} -x^{-r}.
\end{align}
Applying Lemma \ref{totalpopulation} and \eqref{t1:eq4} in \eqref{t1:eq2} yields\[F_I(a_nx)^{Y^I_n} \xrightarrow[ n \rightarrow \infty]{a.s.}  \exp \big\{ -\zeta W x^{-r} \big\}.\] Similarly, for $i\neq I$,
\begin{equation}
\begin{aligned}\label{th:eq1}
   \rho^n \log F_i(a_n x) &\sim \rho^n (1- F_i(a_n x )) \\&= \rho^n (1- F_I(a_n x) ) \frac { (1- F_i(a_n x )) } {(1- F_I(a_n x ))} \sim x^{-r} (a_n x)^{r - r_i} \frac {L_i(a_n x)} {L_I(a_n x)}
\end{aligned}
\end{equation}
which converges to 0, because $r<r_i$ and $L_i$'s are slowly varying. Again, by Lemma \ref{totalpopulation} $\rho^{-n} Y^i_n$ has a finite limit, so \eqref{th:eq1} yields $F_i(a_nx)^{Y^i_n} \xrightarrow[ n \rightarrow \infty] {\mathbb{P}-a.s.} 1$ for $i\neq I$. Hence,
\[
\prod_{i \in \mathcal{C}} F_i(a_nx)^{Y^i_n} \xrightarrow[ n \rightarrow \infty]{a.s.}  \exp \big\{ -\zeta W x^{-r} \big\}
\]
and using the dominated convergence theorem, we have shown 
\begin{equation}\label{maxdisp}
    \mathbb{P}\left(M_n \leq a_n x\right) \xrightarrow[ n \rightarrow \infty]{} \mathbb{E}\left[ \exp \big\{ - \zeta W x^{-r} \big\}\right]
\end{equation}
To finish the proof, we need to show that $\mathbb{P}\left(M_n \leq a_n x\right) \sim \mathbb{P}\left(R_n \leq a_n x\right)$ as $n\rightarrow \infty$. Observe that for any $\varepsilon > 0$ 
    \[
    \mathbb{P} \left( R_n > a_n x\right) \leq \mathbb{P}\left(M_n > a_n (x-\varepsilon)\right) + \mathbb{P} \left( R_n > a_n x, M_n \leq a_n(x-\varepsilon)\right)
    \]
    and 
        \[
    \mathbb{P} ( R_n > a_n x) \geq \mathbb{P}(M_n > a_n (x+\varepsilon)) - \mathbb{P} ( R_n \leq a_n x, M_n > a_n(x+\varepsilon)).
    \]
    Hence, it suffices to show 
    \begin{equation}\label{upperbound}
            \mathbb{P} \left( R_n > a_n x, M_n \leq a_n(x-\varepsilon)\right) \xrightarrow[ n \rightarrow \infty]{} 0 
    \end{equation}
    and 

    \begin{equation}\label{lowerbound}
            \mathbb{P} ( R_n \leq a_n x, M_n \geq a_n(x+\varepsilon))\xrightarrow[ n \rightarrow \infty]{} 0 .
    \end{equation}

We start by showing \eqref{upperbound}. First observe
\begin{align*}
    \mathbb{P} ( R_n > a_n x, M_n \leq a_n(x-\varepsilon)) \leq \E \left[Z_n(a_n x, \infty) \boldone_{M_n \leq a_n(x- \varepsilon)}\right] 
\end{align*}
where $Z_n(a_n x, \infty)$ is the number particles in the $n$-th generation, that are positioned above $a_n x$.  We now need to introduce some new notation:
Denote $F_{i,n}$ for $n$-th convolution of $F_i$ (the distribution function of a sum of $n$ independent random variables distributed as $F_i$). Furthermore, for $\vec{n}= (n_1,n_2, \dots, n_d)$, let \[F_{\vec{n}}(x) = F_{1,n_1} \ast F_{2,n_2} \ast \dots \ast F_{d,n_d}(x)\]
 For a distribution function $F$ and $x, y \in \R$, let $F^y(x) = F(x) \wedge F(y)$ be the distribution function $F$ trimmed at $y$. Note that if $S_n$ is a random walk with step distribution $F$, then 
\begin{equation}\label{t1:eq11}
    F_n^y(x) = \mathbb{P} (S_n < x, \sup\limits_{1 \leq k \leq n} S_k - S_{k-1} < y)
\end{equation}
where $F_n^y$ is the $n$-th convolution of $F^y$.
Now, for a particle in the n-th generation, which had $n_i$ ancestors of type $i$, with $n = \sum_{i \in \mathcal{C}} n_i$, the probability of it ending up in $(a_n x, \infty)$, while all the displacements on the path are smaller than $a_n(x- \varepsilon)$, is 
\[
F_{\vec{n}}^{(a_n (x - \varepsilon))}(\infty) - F_{\vec{n}}^{(a_n (x - \varepsilon))} (a_n x).
\]

Let $A_{\vec{n}}$ be the expected number of particles in the $n$-th generation, that had $n_i$ ancestors of type $i$ for each respective $i \in \mathcal{C}$. Then
\begin{equation}
\begin{aligned}\label{th:integral}
    &\E [Z_n(a_n x, \infty) \boldone_{M_n \leq a_n(x- \varepsilon)}] \leq   \sum_{\vec{n}}A_{\vec{n}}\left(F_{\vec{n}}^{(a_n (x - \varepsilon))}(\infty) - F_{\vec{n}}^{(a_n (x - \varepsilon))} (a_n x)\right)
\end{aligned}
\end{equation}
Here we want to apply inequality (1) from step 3 of the proof in \cite{Durrett1983}. It states that for a regularly varying distribution function $F$ with exponent $r$, all $x,\varepsilon, \delta > 0$ and $ s \in (0,r)$, and a constant $C>0$, we have
\begin{align}\label{eq3}
    F^{(a_{n} (x - \varepsilon))}_{n} (\infty) - F^{(a_{n} (x - \varepsilon))}_{n} (a_{n} x )\leq C \left( \frac {n C_s} {a_{n}^s(x - \varepsilon)^s} \right)^{\frac {x(1-\delta)} {(x - \varepsilon)} }
\end{align}
for all $n$.
where $C_s$ is a constant depending only on $s$. This is not immediately applicable in our case, as $F_{\vec{n}}$ is a convolution consisting of a number of different distributions. However, the statement can be easily generalized to our case as long as all distributions satisfy the requirements. To see that this is true, first note that the aforementioned result in \cite{Durrett1983} is based on a more general bound obtained in the proof of Lemma 3 in \cite{Durrett2}. To generalize the bound to the case with mixed distributions, note that the proof relies on the observation that for $h>0$, and $x>y$, 
\begin{align*}
F^{y}_{n} (\infty) - F^{y}_{n} ( x )\leq R(h,y)^n \exp(-hx)
\end{align*}
where
\[
R(h,y) = \int_{-\infty}^y e^{hu}F^y(du).
\]
The conclusion is then the result of the bounds on $R(h,y)$. In the case of mixed distributions, we can obtain a similar inequality. That is, let $R_i(h,y) = \int_{-\infty}^y e^{hu}F_i^y(du)$. Then
\begin{align}\label{th:eq5}
F^{y}_{\vec{n}} (\infty) - F^y_{\vec{n}} ( x ) = \int_{x}^\infty e^{-hu} e^{hu} dF^y_{\vec{n}} ( u ) \leq e^{-hx} \int_{-\infty}^\infty e^{hu} dF^y_{\vec{n}}(u) = e^{-hx} \int_{-\infty}^{ny} e^{hu} dF^y_{\vec{n}}(u).
\end{align}
Now fix any $i\in\mathcal{C}$. Integrating by parts and exchanging integrals, we get
\begin{align}\label{t1:eq5}
\begin{aligned}
   \int_{-\infty}^{ny} e^{hu} d F^y_{\vec{n}}(u) &= e^{hny} \prod_{j=1}^d F_j(y)^{n_j} - \int_{-\infty}^{ny} F^y_{\vec{n}}(u) d e^{hu} \\&= e^{hny} \prod_{j=1}^d F_j(y)^{n_j} - \int_{-\infty}^{ny} \left[\int_{-\infty}^\infty F^y_{i,n_i}(u-z)  dF^y_{\vec{n}/n_i}(z)\right]d e^{hu} \\&=
    e^{hny} \prod_{j=1}^d F_j(y)^{n_j} - \int_{-\infty}^{(n-n_i)y} \left[\int_{-\infty}^{ny} F^y_{i,n_i}(u-z)  d e^{hu}\right]dF^y_{\vec{n}/n_i}(z)
    \end{aligned}
\end{align}
where 
\[
F^y_{\vec{n}/n_i}(z) =  F_{1,n_1}  \ast \dots \ast F_{i-1,n_{i-1}}\ast F_{i+1,n_{i+1}}\ast \dots \ast F_{d,n_d}(x).
\]
Integrating by parts again,
\begin{align}
\begin{aligned}\label{t1:eq6}
&\int_{-\infty}^{(n-n_i)y} \left[\int_{-\infty}^{ny} F^y_{i,n_i}(u-z)  d e^{hu}\right]dF^y_{\vec{n}/n_i}(z) \\&=  \int_{-\infty}^{(n-n_i)y} \left[e^{hny}F^y_{i,n_i}(ny-z)-\int_{-\infty}^{ny}   e^{hu}dF^y_{i,n_i}(u-z)\right]dF^y_{\vec{n}/n_i}(z) \\&=
\int_{-\infty}^{(n-n_i)y} e^{hny}F^y_{i,n_i}(ny-z)dF^y_{\vec{n}/n_i}(z)-\int_{-\infty}^{(n-n_i)y}\int_{-\infty}^{ny}   e^{hu}dF^y_{i,n_i}(u-z)dF^y_{\vec{n}/n_i}(z) \\&= e^{hny}F^y_{i,n_i}(n_iy)F^y_{\vec{n}/n_i}((n-n_i)y)-\int_{-\infty}^{(n-n_i)y}\int_{-\infty}^{ny}   e^{hu}dF^y_{i,n_i}(u-z)dF^y_{\vec{n}/n_i}(z) \\&=
e^{hny} \prod_{j=1}^d F_j(y)^{n_j} - \int_{-\infty}^{(n-n_i)y}e^{hz}\int_{-\infty}^{ny -z}   e^{hw}dF^y_{i,n_i}(w)dF^y_{\vec{n}/n_i}(z) 
\end{aligned}
\end{align}
The last two equalities are justified by the fact that if $x>ny$, then $F^y_{\vec{n}}(x) = \prod_{j=1}^d F_j(y)^{n_j} $ (see \eqref{t1:eq11}). Similarly, the inner integral in the last line only goes up to $n_iy$, as $z\leq (n-n_i)y$ and $F^y_{i,n_i}(w)$ is constant for $w\geq n_iy$. Using an analogous procedure of integrating by parts, expanding the convolution (this time with respect to $F_i^y$ and $F_{i,n_i-1}^y$), interchanging the integrals, and integrating by parts again, we similarly obtain the following equality.
\begin{align*}
    \int_{-\infty}^{n_i y}   e^{hw}dF^y_{i,n_i}(w) &=  \int_{-\infty}^{(n_i-1)y}e^{hz} \int_{-\infty}^y e^{hw} dF_{i}^y(w)dF_{i,n_i-1}^y(z) \\&= \int_{-\infty}^{(n_i-1)y}e^{hz} R_i(h,y)dF_{i,n_i-1}^y(z)
\end{align*}
 Repeating $n_i-1$ times,
\begin{align*}
\int_{-\infty}^{n_i y}   e^{hw}dF^y_{i,n_i}(w) = R_i(h,y)^{n_i}.
\end{align*}
Plugging this into \eqref{t1:eq6}, we obtain 
\begin{align*}
    &\int_{-\infty}^{(n-n_i)y} \left[\int_{-\infty}^{ny} F^y_{i,n_i}(u-z)  d e^{hu}\right]dF^y_{\vec{n}/n_i}(z) \\&=
    e^{hny} \prod_{j=1}^d F_j(y)^{n_j} - \int_{-\infty}^{(n-n_i)y}e^{hz}R_i(h,y)^{n_i}dF^y_{\vec{n}/n_i}(z)
\end{align*}
and together with \eqref{t1:eq5}, this yields
\begin{align*}
\int_{-\infty}^{ny} e^{hu} d F^y_{\vec{n}}(u) =\int_{-\infty}^{(n-n_i)y}e^{hz}R_i(h,y)^{n_i}dF^y_{\vec{n}/n_i}(z).
\end{align*}
Iterating the whole procedure $d-1$ times to cycle through all types and applying the result in \eqref{th:eq5}, we obtain the following result.
\[
F^{y}_{\vec{n}} (\infty) - F^{}_{\vec{n}} ( x )\leq \exp(-hx) \prod_{i=1}^dR_i(h,y)^{n_i} ,\]
Letting $R(h,y) = \max\limits_{i \in \mathcal{C}} R_i(h,y)$, we can write
\begin{align*}
F^{y}_{\vec{n}} (\infty) - F^{}_{\vec{n}} ( x )\leq \exp(-hx) R(h,y)^{n}.
\end{align*}
Applying to $R(h,y)$ the same bounds as in the proof of Lemma 3 in \cite{Durrett2}, and then the truncation argument from Step 3 of the proof in \cite{Durrett1983} to adapt the result to regularly varying distributions (and noting that $r = \min\limits_{i}r_i$, $s<r$ implies $s<r_i$ for all $i$), we see that \eqref{th:eq5} indeed holds for mixed distributions. Hence,
\begin{align}\label{t1:eq8}
\left(F_{\vec{n}}^{(a_n (x - \varepsilon))}(\infty) - F_{\vec{n}}^{(a_n (x - \varepsilon))} (a_n x)\right) \leq C \left( \frac {n C_s} {a_{n}^s(x - \varepsilon)^s} \right)^{\frac {x(1-\delta)} {(x - \varepsilon)} }
\end{align}
Now and choose $\delta$ small enough so that $\theta = {\frac {x(1-\delta)} {(x - \varepsilon)} } >1$, and take $p \in (r,s)$ satisfying $\frac s p \theta > 1$. The for some $C'>0$ (see Remark \ref{t1:remark}),
\[
C \left( \frac {n C_s} {a_n^s(x - \varepsilon)^s} \right)^\theta \leq C'\rho^{- \frac s p \theta n} \left( \frac {n C_s} {(x - \varepsilon)^s} \right)^\theta
\]
Then for $p' \in (1, \frac s p \theta )$ and suitable $C'' >0$, we have for all $n>0$
\[
C' \rho^{- \frac s p \theta n} \left( \frac {n C_s} {(x - \varepsilon)^s} \right)^\theta   \leq C'' \rho^{-np'}
\]
Ultimately we get 
\begin{align*}
    \left(F_{\vec{n}}^{(a_n (x - \varepsilon))}(\infty) - F_{\vec{n}}^{(a_n (x - \varepsilon))} (a_n x) \right)\leq C'' \rho^{-np'}.
\end{align*}
Hence,

\begin{equation*}
\E [Z_n(a_n x, \infty) \boldone_{M_n \leq a_n(x- \varepsilon)}] \leq C'' \rho^{-n  {p'} }\sum_{\vec{n}}  A_{\vec{n}} = C'' \rho^{-n  {p'} }\mathbb{E}[|Z_n|]= C'' \rho^{-n  {p'} } |M_nZ_0|. 
\end{equation*}
Since $\rho^{-n} M^n$ has a finite limit, and $p'>1$, we get \[  C'' \rho^{-n  {p'} } |M_nZ_0| \xrightarrow[n\rightarrow \infty]{} 0. \] Thus, we have proved \eqref{upperbound}.

Denote by $\eta^n$ one of the nth generation particles that descend from a path on which $M_n$ was attained, by $Q_n$ its position, and by $T(i,n)$ the number of its ancestors of type $i$, excluding the particle attaining $M_n$. Note that $\sum_{i=1}^d T(i,n) = n-1$, and let $\vec{T}(n) = \left(T(1,n), T(2,n), \dots, T(d,n)\right)$. For a distribution function $F$ and $y \in \R$, let $\bar{F}^y(x) = \frac{F^y(x)} {F(y)}$ and denote $\bar{F}^y_n$ the n-th convolution of $\bar{F}^y$. Note that if $S_n$ is a random walk with step distribution $F$, then \[\bar{F}^y_n(x) = \frac{F^y_n(x)} {F(y)^n} = \mathbb{P} (S_n < x|\sup\limits_{1 \leq k \leq n} S_k - S_{k-1} < y).\]

With this notation $Q_n - M_n$ has the following distribution function.
\[ \mathbb{P}(Q_n - M_n \leq x) = \E \left[   \bar{F}^{M_n}_{\vec{T}(n)} (x)  \right], \]
That is, it is distributed as a sum of $T(i,n)$ steps from the distributions $F_i$, respectively, conditioned on the fact that they are all smaller from $M_n$. Now, 
\begin{align*}
\mathbb{P} ( R_n \leq a_n x, M_n &\geq a_n(x+\varepsilon)) \leq \mathbb{P} ( Q_n \leq a_n x, M_n \geq a_n(x+\varepsilon)) \\&= \mathbb{P} ( Q_n - M_n \leq a_n x - M_n, M_n \geq a_n(x+\varepsilon))  \\&\leq \mathbb{P} ( Q_n - M_n \leq a_n x - a_n(x+\varepsilon), M_n \geq a_n(x+\varepsilon)) \\& \leq \mathbb{P} ( Q_n - M_n \leq -a_n \varepsilon) \\&=\E \left[   \bar{F}^{M_n}_{\vec{T}(n)}  (-a_n \varepsilon)  \right] \\&=
\E \left[   \bar{F}^{M_n}_{\vec{T}(n)}  (-a_n \varepsilon) \boldone_{\{ M_n \leq 0 \}} \right] +\E \left[   \bar{F}^{M_n}_{\vec{T}(n)}  (-a_n \varepsilon) \boldone_{\{ M_n > 0 \}} \right] 
\\&\leq \mathbb{P} (M_n \leq 0) + \E \left[  \bar{F}^{0}_{\vec{T}(n)}  (-a_n \varepsilon)  \right].
\end{align*}
Since $ \mathbb{P} (M_n \leq 0) \rightarrow 0$, we only need to take care of the second term. To do this, write
\begin{align}\label{th:eq4}
\E \left[  \bar{F}^{0}_{\vec{T}(n)}  (-a_n \varepsilon)  \right] = \E \left[ \int_{-\infty}^{\infty} \bar{F}^0_{\vec{T}(n)/T(1,n)} (-a_n \varepsilon - y) \bar{F}^0_{T(1,n)}(dy)\right].
\end{align}
Now choose $0 < \delta_1 < \varepsilon$ and split the integral at the point $-a_n \delta_1$. Then
\begin{align}\label{th:int3}
\begin{aligned}
    \E \left[ \int_{-\infty}^{-a_n \delta_1} \bar{F}^0_{\vec{T}(n)/T(1,n)} (-a_n \varepsilon - y) \bar{F}^0_{T(1,n)}(dy)\right] &\leq \E \left[\bar{F}^0_{T(1,n)}(-a_n \delta_1)\right] \\&\leq \E \left[\bar{F}^0_{n}(-a_n \delta_1)\right]
\end{aligned}
\end{align}
and 
\begin{align}\label{th:int4}
    &\E \left[ \int^{\infty}_{-a_n \delta_1} \bar{F}^0_{\vec{T}(n)/T(1,n)} (-a_n \varepsilon - y) \bar{F}^0_{T(1,n)}(dy)\right] \leq \E \left[\bar{F}^0_{\vec{T}(n)/T(1,n)}(-a_n \delta_1)\right] .
\end{align}
Note that the last expression in \eqref{th:int4} is of the same form as the term we started with in \eqref{th:eq4}, except we exchanged $\varepsilon$ for $\delta_1$ and eliminated type $1$. Therefore, applying \eqref{th:int3} and \eqref{th:int4} $d$ times with $ \delta_{d} < \dots < \delta_1  < \varepsilon$, we get
\[
\E \big[  \bar{F}^{0}_{\vec{T}(n)}  (-a_n \varepsilon)  \big] \leq \sum_{i \in \mathcal{C}} \E \left[\bar{F}^0_{n}(-a_n \delta_i)\right]
\]
To see that $\bar{F}^0_{n}(-a_{n} \delta) \rightarrow 0$ as $n\rightarrow \infty$, we refer to Step 4 of the proof in \cite{Durrett1983}. We note that the arguments provided there are based only on the condition $\log(x)F(-x) \rightarrow 0$ as $x\rightarrow \infty$, and the fact that $a_n$ grows exponentially fast, so they are also applicable here. This holds for all $i \in \mathcal{C}$, so by the bounded convergence theorem, the whole expression converges to 0.

This concludes the proof of \eqref{lowerbound}, and thus of the theorem. 
\end{proof}
\subsection{Displacements with semi-exponential tails}

In this section we assume that the displacements are independent and admit semi-exponential tails:\begin{equation}\label{semiexp}
    \mathbb{P}\left(\xi^j \geq t\right) = a_j(t) \exp\{-L_j(t)t^{r_j}\},
\end{equation} where $L_j, a_j$ are slowly varying functions such that $\frac {L_j(t)} {t^{1-r_j}}$ are eventually nonincreasing, and $r_j \in (0,1)$. We also assume that they have finite moments. These assumptions are analogous to the one-type model studied by Gantert in \cite{Gantert2000}. In this section, we show an analogous limit theorem for irreducible multi-type branching random walk. 

Our result is as follows.
\begin{theorem}\label{t2}
Let $r = \min\{r_i : i \in \mathcal{C}\}$, $L(t) = \min \{ L_i(t) : r = r_i\}$, and choose $\psi(n)$ to be a positive function satisfying
\begin{equation}\label{t2:norm}
\frac {L(\psi(n))\psi(n)^{r}} n \rightarrow 1.
\end{equation}
Then
\[
\frac {R_n} {\psi(n)} \xrightarrow{a.s.}  (\log \rho)^{\frac 1 {r}}.
\]
\end{theorem}
\begin{remark}\label{t2:remark}
    \emph{As in the one type case, the existence of $\psi(n)$ satisfying \eqref{t2:norm} is guaranteed by the result of de Bruijn \cite{DeBruijn}. Indeed, if $K(x)$ is the de Bruijn conjugate of $x \mapsto L\left(x^{\frac 1 r}\right)$, then we can take $\psi(n)=K(n)^{\frac 1 r}n^{\frac 1 r}$. In particular, this implies that for any $\varepsilon > 0$, \[
    n^{\frac 1 r (1-\varepsilon)} \leq \psi(n) \leq n^{\frac 1 r (1+\varepsilon)}
    \]
    for large enough $n$.
    }
\end{remark}
    
    We will also show the following lemmas, which describe the asymptotic behavior of the underlying multi-type Galton-Watson process.
         \begin{lemma}\label{lemma}
         Let $|Z_n| = \sum_{r \in \mathcal{C}} Z^r_n$ be the sum of all particles in the $n$-th generation of the process. Then for any $\varepsilon > 0$ there is $ 0< \delta \leq \varepsilon$ satisfying
         \[
         \mathbb{P} \left(\frac{|Z_n|} {\rho^n}  <  {(1-\varepsilon)^n} \right) < (1-\delta)^{ n}
         \]
         for all $n$ large enough. 
     \end{lemma}

     \begin{lemma}\label{lemma2}
     There exist $\delta >0$ and $\beta \in (0,1)$, such that for all $i \in \mathcal{C}$ and all $n$ large enough 
     \[
     \mathbb{P}(Z_n^i < \delta |Z_{n-l}|) \leq \beta^n,
     \]where $l \in \mathbb{N}$ is as in \eqref{as3}.
     \end{lemma}
    
     \begin{proof}[Proof of the lemma \ref{lemma}.]
         First note that if $\rho(1-\varepsilon) \leq1$, then $\mathbb{P} \left(\frac{|Z_n|} {\rho^n}  <  {(1-\varepsilon)^n} \right) = 0$ and the statement is trivial. Assume $\rho(1-\varepsilon) >1$. The key tool to proving the lemma is the result of Athreya and Vidyashankar (\cite{Athreya1995}, Theorem 2.6), which we state below.
         \begin{lemma}\label{athreyaeq}
         Additionally to our standing assumptions, assume
\begin{equation}\label{athreyaexp}
    \text{there exists } \theta_0 > 0 \text{, such that }\mathbb{E}_i \left[\exp \{\theta_0 Z_1^j \}\right] < \infty \text{ for all } i,j \in \mathcal{C}.
\end{equation}
Then there are constants $C>0$, $\lambda > 0$ such that for any $\varepsilon > 0$
         \begin{equation}
             \mathbb{P} \left( \Big|\frac{Z_n} {\rho^n} \cdot v - W \Big| \geq  \varepsilon \right) < C \exp\{-\lambda \left(\varepsilon^2 \rho^n \right)^{\frac 1 3}\}
         \end{equation}
         for all $n$.
         \end{lemma}
 Although its stated in \cite{Athreya1995} for two types, it is clear from the proof that the same argument holds for an arbitrary number of types. 
 Another important inequality we will use is a straightforward consequence of results from Jones \cite{jones} describing the small-value probabilities of $W$: there exists $\alpha >0$, such that for small enough $\varepsilon >0$,
\begin{equation}\label{joneseq}\mathbb{P}(W \leq \varepsilon)\leq \varepsilon^\alpha.
\end{equation}
Since we do not assume the existence of exponential moments $\eqref{athreyaexp}$, some additional steps are required to use Lemma $\ref{athreyaeq}$. Consider a trimmed Galton-Watson process $Z_n(L)$ generated by random variables $N_{i,j}{(L)} = N_{i,j} \boldone_{\{ N_{i,j} < L\}} $ for some $L>0$ and denote by $M{(L)}$ its mean matrix, by $\rho{(L)}$ its largest eigenvalue, and by $v(L)$ its left eigenvector. Since $\rho{(L)} \rightarrow \rho$ as $L \rightarrow \infty$, we choose $L$ large enough so that $\rho{(L)} > (1-\varepsilon)\rho$. Clearly, 
         \[
         \mathbb{P} \left(\frac{|Z_n|  }{\rho^n}  <  {(1-\varepsilon)^n} \right) \leq \mathbb{P} \left(\frac{|Z_n(L)|} {\rho^n} <  {(1-\varepsilon)^n} \right).
         \]
         Now choose $\delta >0$ satisfying
         \[
         (1 - \delta)\rho{(L)} \geq (1 - \varepsilon) \rho, \hspace{3mm} (1-\delta)^2\rho(L) > 1
         \]
         so that 
         \[ \mathbb{P} \left(\frac{|Z_n(L)|} {\rho^n} <  {(1-\varepsilon)^n} \right) \leq  \mathbb{P} \left(\frac{|Z_n(L)|} {\rho(L)^n} <  {(1-\delta)^n} \right)\]
         Since $Z_n(L)$ satisfies the same assumptions we make on $Z_n$ in this chapter, then \eqref{limit} holds for some $W(L)$ and \eqref{joneseq} holds for appropriate choice of constants. Note that $|Z_n(L)| \geq \frac 1 {||v||_\infty}Z_n(L)\cdot v(L)$, and observe 
        \begin{align*}
        &\mathbb{P} \left(\frac{|Z_n(L)|} {\rho(L)^n}  <  {(1-\delta)^n} \right) \leq \mathbb{P} \left(\frac{Z_n(L) \cdot v(L)} {\rho(L)^n}  <  ||v(L)||_\infty  {(1-\delta)^n} \right)  \\ &= \mathbb{P} \left(\frac{Z_n(L) \cdot v(L)} {\rho(L)^n}  <  ||v(L)||_\infty  {(1-\delta)^n}, \hspace{2mm} W(L)  \geq  \frac {3 ||v(L)||_\infty } 2{(1- \delta )^n}\right) \\ &+ \mathbb{P} \left(\frac{Z_n(L) \cdot v(L)} {\rho(L)^n}  <  ||v(L)||_\infty  {(1-\delta)^n}, \hspace{2mm} W(L)  <  \frac {3 ||v(L)||_\infty } 2{(1- \delta )^n}\right)
        \\&\leq \mathbb{P} \left( \Bigg|\frac{Z_n(L)\cdot v(L)} {\rho(L)^n} - W(L) \Bigg| \geq  \frac {||v(L)||_\infty } 2{(1- \delta )^n}\right)
        \\&+ \mathbb{P} \left(  W(L)  <  \frac {3 ||v(L)||_\infty } 2{(1- \delta )^n} \right)
        \end{align*}
        By \eqref{joneseq}, the second term is bounded by $\left(\frac {3  ||v(L)||_\infty } {2}\right)^{\alpha} (1- \delta )^{\alpha n} $ for $\alpha >0$ and $n$ large enough.
        The bound for the first term follows from Lemma \ref{athreyaeq}. Since $Z_n(L)$ satisfies \eqref{athreyaexp} as well as our standing assumptions on $Z_n$, we conclude from \ref{athreyaeq}, that for appropriate $C > 0$, $\lambda > 0$
        \[
        \mathbb{P} \left( \Bigg|\frac{Z_n(L) \cdot v(L)} {\rho(L)^n} - W(L) \Bigg| \geq  \frac {||v(L)||_\infty} 2{(1- \delta )^n}\right) \leq C \exp\left\{-\lambda \left((1-\delta)^2 \rho(L) \right)^{\frac n 3} \right\},
        \]
        and for large enough $n$, 
        \[
        C \exp\left\{-\lambda \left((1-\delta)^2 \rho(L) \right)^{\frac n 3} \right\} \leq (1-\delta)^{\alpha n}.
        \]
        Since $ \left(1 + \left(\frac {3  ||v(L)||_\infty } {2}\right)^{\alpha} \right)(1-\delta)^{\alpha n}  < (1-\delta_0)^n$ for some $\delta_0 < \delta$ and large enough $n$, the lemma is proven.
        \end{proof}

        \begin{proof}[Proof of the Lemma \ref{lemma2}]
        Fix $i \in \mathcal{C}$ and denote by $Z_l^{r \rightarrow i}$ a generic random variable distributed as $Z_l^i$ under $\mathbb{P}_r$. Recall that $l$ is a natural number for which $M^l$ has only strictly positive entries (see  assumption \eqref{as3}). Consequently, $q_{r,i} = \mathbb{P} \left( Z_l^{r \rightarrow i} = 0 \right) < 1$ for all $r \in \mathcal{C}$. Let $q_{\max} = \max \{ q_{r,i}: r \in \mathcal{C}\}$ and $q_{\min} = \min \{ q_{r,i}: r \in \mathcal{C}\}$. Then
        \begin{align*}
            \mathbb{P}\left(Z_n^i < \delta |Z_{n-l}|\right) = \mathbb{P}\left(\sum_{r \in \mathcal{C}} \sum_{m=1}^{Z_{n-l}^r} Z_l^{r \rightarrow i}(m)   < \delta |Z_{n-l}|\right) = \mathbb{E} \left[ \Phi(Z_{n-l})\right]
        \end{align*}
        where for fixed $r$, $\{Z_l^{r \rightarrow i}(m)\}_{m \geq 1} $ are independent copies of $Z_l^{r \rightarrow i}$, and for $r_1\neq r_2$, $\{Z_l^{r_1 \rightarrow i}(m)\}_{m \geq 1} $ are independent of $\{Z_l^{r_2 \rightarrow i}(m)\}_{m \geq 1} $, and 
        \[\Phi(k) = \mathbb{P}\left(\sum_{r \in \mathcal{C}} \sum_{m=1}^{k_r} Z_l^{r \rightarrow i}(m)   < \delta |k|\right)\]
        for $k = (k_1, \dots, k_d) \in \mathbb{N}^d$.
        For a vector $j = (j_1, \dots, j_d) \in \mathbb{N}^d$ satisfying $j_r \leq k_r$ for all $r \in \mathcal{C}$, define \[
        A_{(j_1, \dots, j_d)} =  \bigcap_{r \in \mathcal{C}}\left\{ \Big|\left\{ m \leq k_r: Z_l^{r \rightarrow i}(m) > 0 \right\}\Big| = j_r \right\}.
        \]
        Observe 
        \begin{align*}\mathbb{P}(A_{(j_1, \dots, j_d)}) &= \prod_{r \in \mathcal{C}} \binom{k_r} {j_r} q_{r,i}^{k_r - j_r} \left(1 - q_{r,i}\right)^{j_r} \leq  \prod_{r \in \mathcal{C}} \binom{k_r} {j_r} q_{\max}^{k_r - j_r} \left(1 - q_{\min}\right)^{j_r}  \\&= \prod_{r \in \mathcal{C}} \binom{k_r} {j_r}  q_{\max}^{k_r} \left(\frac {1 - q_{\min}} {q_{\max}}\right)^{j_r} = q_{\max}^{|k|} C^{|j|} \prod_{r \in \mathcal{C}} \binom{k_r} {j_r}\end{align*}
    where $C = \frac {1 - q_{\min}} {q_{\max}}$. By the generalized Vandermonde's identity, for any $j \in \mathbb{N_+}$
    \[
    \sum_{j_1 + \dots + j_d = j} \prod_{r \in \mathcal{C}} \binom{k_r} {j_r} = \binom{|k|} {j}
    \]
    where the sum $\sum_{j_1 + \dots + j_d = j}$ goes over all partitions of $j$, and we put $\binom{n}{m} = 0$ if $m>n$. Hence
    \begin{align*}
        \Phi(k) &= \sum_{j = 1}^\infty \sum_{j_1 + \dots + j_d = j} \mathbb{P}\left(\sum_{r \in \mathcal{C}} \sum_{m=1}^{k_r} Z_l^{r \rightarrow i}(m)   < \delta |k|, \hspace{2mm} A_{j_1, \dots, j_2}\right) \\&\leq \sum_{j = 1}^\infty \sum_{j_1 + \dots + j_d = j} \mathbb{P}\left(j   < \delta |k|, \hspace{2mm} A_{j_1, \dots, j_2}\right) = \sum_{j = 1}^{\lfloor\delta|k|\rfloor} \sum_{j_1 + \dots + j_d = j} \mathbb{P}\left(A_{j_1, \dots, j_2}\right) 
        \\&\leq\sum_{j = 1}^{\lfloor\delta|k|\rfloor} \sum_{j_1 + \dots + j_d = j} q_{\max}^{|k|} C^{j} \prod_{r \in \mathcal{C}} \binom{k_r} {j_r}  = q_{\max}^{|k|}\sum_{j = 1}^{\lfloor\delta|k|\rfloor} C^{j}\binom{|k|} {j} \\&\leq q_{\max}^{|k|}\sum_{j = 1}^{\lfloor \delta|k|\rfloor} C^{j}\frac {|k|^j} {j !} = q_{\max}^{|k|}\sum_{j = 1}^{\lfloor \delta|k|\rfloor} \left(\frac {C}{\delta}\right)^{j}\frac {(\delta|k|)^j} {j !}.
    \end{align*}
    Choosing $\delta$ so that $\frac {C} \delta > 1$, we have
    \begin{align*}
        q_{\max}^{|k|}\sum_{j = 1}^{\lfloor \delta|k|\rfloor} \left(\frac {C}{\delta}\right)^{j}\frac {(\delta|k|)^j} {j !} &\leq q_{\max}^{|k|} \left(\frac {C}{\delta}\right)^{\lfloor \delta|k|\rfloor} \sum_{j = 1}^{\lfloor \delta|k|\rfloor} \frac {(\delta|k|)^j} {j !} \leq q_{\max}^{|k|} \left(\frac {C}{\delta}\right)^{\delta|k|} e^{\delta|k|} \\&= \left(q_{\max}\left( \frac {C e}{\delta}\right)^{\delta} \right)^{|k|} 
    \end{align*}
    Since $\left( \frac {C e}{\delta}\right)^{\delta} \rightarrow 1$ as $\delta \rightarrow 0$ and $q_{\max} < 1$, choosing $\delta$ small enough we have 
    \[\Phi(k) \leq \beta_0^{|k|}
    \]
    for $\beta_0 < 1$. Hence
    \[
    \mathbb{P}\left(Z_n^i < \delta |Z_{n-l}|\right) = \mathbb{E}[\Phi(Z_{n-l})] \leq \mathbb{E}[\beta_0^{Z_{n-l}}] \leq \beta_0^n + \mathbb{P}(Z_{n-l} < n)
    \]
    Since for any $\varepsilon$ satisfying $\rho(1-\varepsilon) > 1$ we have $n < \rho^n(1-\varepsilon)^n$ for large enough $n$, the bound 
    \[
    \mathbb{P}(Z_{n-l} < n) < C_1\beta_0^n
    \]
    is a straightforward conclusion from Lemma \ref{lemma}. Hence
    \[
    \mathbb{P}\left(Z_n^i < \delta |Z_{n-l}|\right) \leq (1+C_1)\beta_0^n \leq \beta^n 
    \]
    for some $\beta < 1$ and all large enough $n$.
        \end{proof}
        \begin{proof} [Proof of the Theorem]
    We start with the upper bound. Let $\eta$ be a random variable with the distribution function \[
    F(x) = \begin{cases}
        1 - \max\left\{a_i(x) \exp\{-L_i(x)x^r\} \ : \ i \in \mathcal{C}, r_i = r \right\} \hspace{0.5cm}   &x>0\\
        0  & x \leq 0
    \end{cases}
    \] By choice of $r$ and $L$, there exists a constant $c>0$ such that for all $t>c$ and $i \in \mathcal{C}$
    \[
    \mathbb{P} (\eta \geq t) \geq \mathbb{P} (\xi^i \geq t).
    \]
    Hence, $\eta^c = \eta \boldone_{ \{ \eta > c\} } + c \boldone_{ \{ \eta \leq c\} }$ dominates stochastically $\xi^{i,c} = \xi^i \boldone_{ \{ \xi^i > c\} } + c \boldone_{ \{ \xi^i \leq c\} }$ for all $i \in \mathcal{C}$. Since stochastic dominance is preserved under convolution, we have that for any $x>0$, $n \in \mathbb{N}$ and $v \in \mathbb{T}_n$
    \[
     \mathbb{P} \left( \sum_{k =1}^{n} \eta^c_k  \geq x \right)\geq  \mathbb{P} \left( \sum_{k =1}^{|v|} \xi^c_{v_k}  \geq x \right),
    \]
    where $\{\eta^c_k\}_{k \geq 0} $ are i.i.d. distributed as $\eta^c$. Then
    \[
    \mathbb{P}(S_v \geq \psi(n) x) \leq \mathbb{P} \left( \sum_{k =1}^{|v|} \xi^c_{v_k}  \geq \psi(n) x \right) \leq \mathbb{P} \left( \sum_{k =1}^{n} \eta^c_k  \geq  \psi(n)x \right).
    \]
    Note that 
    \[
      a_{min}(x)\exp\{-L(x)x^r\} \leq1-F(x) \leq a_{max}(x)\exp\{-L(x)x^r\} 
    \] where 
    \begin{align*} a_{min}(x) = \min\left\{a_i(x)\ : \ i \in \mathcal{C}, r_i = r \right\}, \\ a_{max}(x) =  \max\left\{a_i(x)\ : \ i \in \mathcal{C}, r_i = r \right\}. \end{align*}
    Since $a_{min}$, $a_{max}$, and $L$ are all slowly varying, Theorem 3 along with the remark (see the equation (29)) from \cite{Gantert2000} asserts that for all $x>0$
    \begin{equation}\label{gantertlemma}
        \lim \frac 1 n \log \mathbb{P} \left(\sum_{k =1}^{n} \eta^c_k \geq \psi(n) x \right) = -x^r.
    \end{equation}
    In particular, for any $\varepsilon > 0$, there is $\delta >0 $ such that for all large enough $n$,
    \begin{equation}
        \mathbb{P}\left(S_v \geq \psi(n) (\log \rho + \varepsilon)^{\frac 1 r}\right) \leq \exp \left\{ -n(\log \rho + \delta) \right\} = \rho^{-n} e^{ -n \delta}
    \end{equation}
    for any $v \in \mathbb{T}_n$.
Having this bound, we proceed as in \cite{Gantert2000}. We have
\begin{equation}\label{calc}
\begin{aligned}
    &\mathbb{P}\left(\exists v \in \mathbb{T}_n : S_v \geq \psi(n) (\log \rho + \varepsilon)^{\frac 1 r}\right) \\ &=\sum_{k=1}^\infty \mathbb{P}\left( \exists v \in \mathbb{T}_n : S_v \geq \psi(n) (\log \rho + \varepsilon)^{\frac 1 r} \Big| |Z_n| = k \right) \mathbb{P} (|Z_n| = k) \\ 
    &\leq \sum_{k=1}^\infty k \rho^{-n} e^{ -n \delta} \mathbb{P} (|Z_n| = k) = 
    \mathbb{E}[|Z_n|] \rho^{-n} e^{ -n \delta},
\end{aligned}
\end{equation}
where $|Z_n| = \sum_{i \in \mathcal{C}} Z^i_n$.
It is easily verifiable by induction, that 
\[
\mathbb{E}[Z_n] = M^n \mathbb{E}[Z_0],
\]
where $M$ is the mean matrix. Hence, $ \mathbb{E}[Z_n] \rho^{-n}$ has a limit, and by linearity so does $\mathbb{E}[|Z_n|] \rho^{-n}$. Applying the Borel-Cantelli lemma to \eqref{calc} and letting $\varepsilon \rightarrow 0$ entails the upper bound in Theorem \ref{t2}.

The lower bound requires more delicate approach. For $K > 0$, let \[\mathbb{T}^K = \{v \in \mathbb{T} : \forall_{k \leq |v|} \xi_{v_k} \geq -K\},\] and $M^K_n = \max \{\xi_v: |v|=n, v \in \mathbb{T}^K\}$, and denote by $\rho_K$ the Perron-Frobenius eigenvalue of the matrix $\{\mathbb{E}[N_{i,j}]\mathbb{P}(\xi^j > -K)\}_{i,j \in \mathcal{C}}$. Since $\rho_K \rightarrow \rho > 1$ as $K \rightarrow \infty$, choose $K$ large enough so that $\rho_K > 1$. Note that 
\begin{align}
    R_n = \max_{|v|=n} S_v \geq \max_{|v|=n, v \in \mathbb{T}^K} S_v \geq M^K_n - (n-1)K.
\end{align}
By Remark \ref{t2:remark}, $\frac {n-1} {\psi(n)} \rightarrow 0$, hence dividing by $\psi(n)$ and taking limits yields 
\begin{align}
\liminf_{n \rightarrow \infty} \frac {R_n} {\psi(n)} \geq \liminf_{n \rightarrow \infty}\frac {M^K_n} {\psi(n)}.
\end{align}
    Hence, it suffices to show 
    \begin{align}\label{max1} \liminf_{n \rightarrow \infty}\frac {M^K_n} {\psi(n)} \geq  (\log \rho_K)^{\frac 1 {r}}
    \end{align}
    By the Borell-Cantelli lemma, to show \eqref{max1} it is enough to show that for any $\varepsilon > 0$
    \begin{align}\label{max} \sum_{n=0}^\infty \mathbb{P} \left( \frac {M^K_n} {\psi(n)} < \left[\log \left\{\rho_K(1-\varepsilon)\right\}\right]^{\frac 1 {r}} \right) <\infty.
    \end{align}
    To that end, take any $\varepsilon > 0$ small enoguh to satisfy $\rho_K(1-\varepsilon) >1$ and let \[Z^{K,i}_n = \#\{v \in \mathbb{T}^K: \sigma(v) =i, |v|=n\},\] and $Z^K_n = \sum_{i \in \mathcal{C}} Z^{K,i}_n$. To simplify the notation, denote \[b_n = \psi(n) \left[\log \left\{\rho_K(1-\varepsilon)\right\}\right]^{\frac 1 {r}}\] and let \[
    I(n) = \argmax_{i \in \mathcal{C}, r_i = r} L_i\left( b_n \right) .
    \]  Then
    \begin{align*}
        &\mathbb{P} \left( \frac {M^K_n} {\psi(n)} < \left[\log \left\{\rho_K(1-\varepsilon)\right\}\right]^{\frac 1 {r}} \right) = \mathbb{E} \left[ \prod_{i \in \mathcal{C}}\mathbb{P} \left( \xi^i < b_n  \right)^{Z^{K,i}_n} \right] \\&=
        \mathbb{E} \left[\prod_{i \in \mathcal{C}} \left(1- \mathbb{P} \left( \xi^i \geq b_n  \right)\right)^{Z^{K,i}_n} \right] \leq
        \mathbb{E} \left[ \prod_{i \in \mathcal{C}}\exp \left\{ -{Z^{K,i}_n} \mathbb{P} \left( \xi^i \geq b_n  \right) \right\} \right] \\&\leq
        \mathbb{E} \left[ \exp \left\{ -{Z^{K,I(n)}_n} \mathbb{P} \left( \xi^{I(n)} \geq b_n  \right) \right\} \right]
    \end{align*}
    Using our assumption \ref{semiexp}, we have
    \begin{align*}
        &\mathbb{E} \left[ \exp \left\{ -{Z^{K,I(n)}_n} \mathbb{P} \left( \xi^{I(n)} \geq b_n  \right) \right\} \right] \\&=
        \mathbb{E} \left[ \exp \left\{ -{Z^{K,I(n)}_n} a_{I(n)}\left(  b_n \right)  \left(\rho_K(1-\varepsilon)\right)^{- L_{I(n)}\left( b_n \right) \psi(n)^r} \right\} \right]
    \end{align*}
    Note that $L_{I(n)}\left( b_n \right) = L\left( b_n \right)$, hence by choice of $L$ (see \eqref{t2:norm}), 
    \[\frac{L_{I(n)}\left( b_n \right) \psi(n)^r} n \rightarrow 1, \]and recall that all $a_i$ are slowly varying. Hence we have 
    \[
    a_{I(n)}\left(  b_n \right)  \left(\rho_K(1-\varepsilon)\right)^{- L_I\left( b_n \right) \psi(n)^r} \geq \left(\rho_K(1-\varepsilon_1)\right)^{-n}
    \]
    for some $\varepsilon_1 \in (0,\varepsilon)$ and all sufficiently large $n$.
    Hence, by Borel-Cantelli lemma, \eqref{max} will follow from the convergence of the series 
    \[
    \sum_{n=1}^\infty \mathbb{E} \left[ \exp \left\{ -\frac {Z^{K,I}_n} {\left(\rho_K(1-\varepsilon_1)\right)^n} \right\}\right].
    \]
     Using the formula 
    \[
     \mathbb{E} \left[ \exp \left\{ -\frac {Z^{K,I}_n} {\left(\rho_K(1-\varepsilon_1)\right)^n} \right\}\right] \leq  \exp \left\{ -\frac {\left(1-\frac {\varepsilon_1} 2\right)^n} {\left(1-\varepsilon_1\right)^n} \right\} + \mathbb{P}\left(\frac {Z^{K,I}_n} {\rho_K^n} \leq \left(1-\frac {\varepsilon_1} 2\right)^n \right)
    \]
    we see that it is sufficient to show 
    \begin{align*}\label{sum}
        \sum_{n=1}^\infty \mathbb{P}\left(\frac {Z^{K,I}_n} {\rho_K^n} \leq \left(1-\frac {\varepsilon_1} 2\right)^n \right) < \infty.
    \end{align*}
    and this is a straightforward consequence of applying Lemmas \ref{lemma2} and \ref{lemma}.
\end{proof}

\section{Reducible multi-type branching random walk}
In this chapter, we consider the case where the mean matrix $M$ is reducible. We divide $\mathcal{C}$ by the following equivalence relation: $i\sim j$ if there are $l_1$, $l_2$ such that $M^{l_1}(i,j) >0$ and $M^{l_2}(j,i) > 0$. We denote $\mathcal{C}_\sim = \{\mathcal{C}_1,\mathcal{C}_2, \dots, \mathcal{C}_m\}$ and introduce a partial ordering of $\mathcal{C}$ through the following relation: $i \preceq j$ if there exists $n \in \mathbb{N}$ such that $M^n(i,j) >0 $. This induces a partial ordering of $\mathcal{C}_\sim$. We will abuse the notation and write $a \preceq b$ when there exist $i \in \mathcal{C}_a$ and  $j \in \mathcal{C}_b$  such that $i \preceq j$, and $a \preceq i$ if $i \in \mathcal{C}_b$ and $a \preceq b$. By renumbering the types, we may and will assume that $M$ is of form
\begin{equation}\label{matrix}
    \begin{pmatrix}
        M[1] & M[1,2] & \dots & M[1,m] \\
        0 & M[2] & \dots & M[2,m]\\
        \vdots & \vdots & \vdots & \cdots \\
        0 & \dots & 0 & M[m]
    \end{pmatrix}
    .
\end{equation}
 That is, it has cages $\{M[a]\}_{a \in \mathcal{C}_\sim}$ on the diagonal and zeros below.  Throughout the chapter, we make the following assumptions. 

\begin{equation}\label{as4}
    M[a] \text{ is positively regular in the sense of }\eqref{as3}\text{ for all } a \leq m.
\end{equation}
 For any $a \leq m$, denote by $\rho(a)$ the largest eigenvalue of $M[a]$ and assume
  \begin{equation}\label{as5}
      \rho(1) >1 
  \end{equation}
  It is easy to see that the spectrum of $M$ is just a union of spectrums of $M[a]$'s.

 It is clear that when the starting particle comes from $\mathcal{C}_a$ the problem is reduced to the analysis of the types of classes following (and including) $a$, hence without loss of generality we may assume that the starting particle's type belongs to class 1. Since the specific type will be of little significance, we will assume for simplicity that the starting particle is of type 1. Similarly, if some class does not follow the first class, it will never appear in the process, so we assume $1 \preceq a$ for all $a \leq m$. To avoid conditioning on the survival set, we assume that type 1 (or equivalently class 1) survives with probability 1. 

    Analogously to the previous section, we assume the following Kesten-Stigum condition: \begin{equation}\label{as6}
       \text{for all } a \leq m\text{ and all pairs } i,j \in \mathcal{C}_a,   \   \mathbb{E}[N_{i,j} \log N_{i,j}] < \infty.
   \end{equation}
   By the result of Kesten and Stigum \cite{kesten2}, under this assumptions, if \[\rho_j = \max{\{\rho_a \ : \ a \in \mathcal{C}_{\sim}, a \preceq j}\}\] then for some $k>0$,
   \begin{align}\label{kestenr}
   \frac {Z^j_n} {n^k\rho_j^n} \xrightarrow{\mathbb{P}-a.s.} W(j), 
   \end{align}
   and $W(j)$ is positive if $u^\alpha_j$ is positive, where $u^\alpha$ is the left eigenvector of $M[\alpha]$. In other words, the asymptotic number of particles of any given type is driven by the number of particles preceding it. More explicit expressions can be provided for $W(j)$ in certain examples, but a general formula seems difficult to obtain. One can show that the randomness in $W(j)$ is contained in the preceding classes that have the highest number of offspring, that is, if for some classes $\alpha$ and $\beta$ we have $\rho_\beta  < \rho_\alpha = \max{\{\rho_\gamma \ : \ \gamma \preceq \beta}\}$, then $W(\beta) = (W(j))_{j \in \mathcal{C}_\beta}$ is a deterministic linear transformation of $W(\alpha)$. We refer to \cite{kesten2} for a more detailed exploration of the properties of $W$. 

\subsection{Displacements with regularly varying tails}
Let $F_i(x) = \mathbb{P}\left(\xi^i \leq x\right)$. In this section, analogously to the irreducible case, we assume that the displacements are independent and there exist slowly varying functions $\{L_i\}_{i \in \mathcal{C}}$ and positive constants $\{r_i\}_{i \in \mathcal{C}}$, satisfying
\begin{align}\label{t3:tails}
\begin{aligned}
     &1-F_i(x) \sim L_i(x) x^{-r_i} \hspace{0.5cm} \text{as }x \rightarrow \infty, \\
     &\log(-x)F_i(x)\rightarrow 0 , \hspace{0.5cm}  \text{as }x \rightarrow -\infty.
\end{aligned}
\end{align}

   In contrast to the assumptions of Theorem \ref{t1}, this time we assume that there is a unique pair $(\alpha,I)$ satisfying \begin{align*}
   &\rho_{\alpha}^{\frac 1 {r_I}} = \max{\{\rho_a^{\frac 1 {r_i}} \ : \ a \preceq i \}}
   \end{align*}
   It is perhaps worth noting that unlike in the irreducible model, $r_I$ is not necessarily the minimum of all $r_i$'s, nor is $\rho$ the principal eigenvalue of $M$. This is due to the fact that the growth speed of a single cage $a$, just as we have seen in Theorem \ref{t1}, is exponential at the rate $\rho_a^{\frac 1 {r_a}}$, where $r_a$ is the minimal exponent among the types from this cage. In other words, the speed depends on the interplay between the tails of the displacements and the asymptotic expected number of particles. Since in the reducible case the latter may be different for different classes, choosing the "dominant" type, and therefore the correct normalization, requires us to look at both of these quantities. 
   
   We denote $\rho=\rho_{\alpha}$ and $r = r_I$.
    Furthermore, let $k >0$ be the constant satisfying \[
   \frac {Z^I_n} {n^k\rho^n} \xrightarrow{\mathbb{P}-a.s.} W(I) 
   \]

 Our result is as follows.
\begin{theorem}\label{t3} Let \[\zeta =\sum_{j>0} \rho^{-j} \sum_{l \in \mathcal{C}}    \mathbb{P}_I\left(Z_j^l> 0\right).\] and choose the sequence $\{a_n\}_{n \in \N}$ so that 
\begin{equation*}\label{an2}
    n^k\rho^n \left(1-F_I(a_n)\right) \xrightarrow[ n \rightarrow \infty]{} 1.
\end{equation*}
Then
\[
\mathbb{P} ( R_n \leq a_n x)  \xrightarrow[ n \rightarrow \infty]{} \E[e^{-\zeta W(I)x^{-q}}]
\]
\end{theorem}
\begin{remark}\label{t3:remark}\emph{
    As in the irreducible case, the existence of $a_n$ satisfying \eqref{t1:an} is guaranteed by the result of de Bruijn \cite{DeBruijn}. Here if $L^\#$ is the de Bruijn conjugate of $L$, we can take $a_n = L^\#\left(n^{\frac k r}\rho^{\frac n {r}}\right) n^{\frac k r}\rho^{\frac n {r}}$. In particular, this guarantees that for any $\varepsilon > 0$, \begin{equation}
        \rho^{\frac n r (1-\varepsilon)} < a_n < \rho^{\frac n r (1+\varepsilon)}
    \end{equation}
    for sufficiently large $n$.
}\end{remark}

\begin{proof}[Proof of the Theorem \ref{t3}]
    Similarly to the irreducible case, we begin with a lemma on the total population. Recall that for $i \in \mathcal{C}$,  \[
    Y^i_n = \left|\bigcup_{k =1}^{n}\left\{ v \in \mathbb{T}_k \ : \ \sigma(v) =i, (\exists w \in \mathbb{T}_n)(w_k = v) \right\}\right|
    \]
is the total number of particles of type $i$ that have offspring in the $n$-th generation.
\begin{lemma}\label{t3:lemmatotal}
    Assume $\eqref{as4}$ and $\eqref{as5}$ and let  \[\zeta_i =\sum_{j>0} \rho^{-j} \sum_{l \in \mathcal{C}}    \mathbb{P}_i\left(Z_j^l> 0\right).\] Then for all $i \in \mathcal{C}$,
    \[
    \frac{Y^i_n} {\rho_i^n n^{k_i}} \rightarrow \zeta_i W(i).
    \]
    where $\rho_i, k_i, W(i)$ are as in \eqref{kestenr}.
\end{lemma}

\begin{proof}
The Lemma is proven with analogous arguments as in Lemma \ref{totalpopulation}, but we provide full argument for convenience of the reader. First decompose
\begin{align*}
    \frac {Y^i_n} {\rho^n n^{k_i}} = \sum_{j=0}^{n-1} \rho^{-j}  \sum_{l \in \mathcal{C}} \left(\frac{n-j} {n} \right)^{k_i}\frac {Z^i_{n-j}} {\rho^{n-j}(n-j)^{k_i}}  \frac 1 {Z^i_{n-j}} \sum_{k=1}^{Z_{n-j}^i }\boldone_{\{Z^l_j{(i,k)} >0\}}
\end{align*}
where for any $j$ and $l$, $\{Z^l_j{(i,k)}\}_{k>0} $ are i.i.d. distributed as $Z^l_j$ under $\mathbb{P}_i$, and for $i_1\neq i_2$, $\{Z^l_j{(i_1,k)}\}_{k>0} $ are independent of $\{Z^l_j{(i_2,k)}\}_{k>0} $. Now we denote \[D_{n-j} = \left(\frac{n-j} {n} \right)^{k_i}\frac {Z^i_{n-j}} {\rho^{n-j}(n-j)^{k_i}}  \] and \[E^l_{n-j} = \frac 1 {Z^i_{n-j}} \sum_{k=1}^{Z_{n-j}^i }\boldone_{\{Z^l_j{(i,k)} > 0\}}\] By the strong law of large numbers, for any fixed $j>0$,\[E^l_{n-j}  \xrightarrow[ n \rightarrow \infty]{a.s.} \mathbb{P}_i(Z^l_j > 0) \] We also know from \eqref{kestenr}, that for any fixed $j>0$, \[D_{n-j}  \xrightarrow[ n \rightarrow \infty]{a.s.}W(i)\]
Now fix $N>0$. Then for $n>N$
\begin{align*}
     \frac {Y^i_n} {\rho^n n^{k_i}} &\leq \sum_{j=0}^{N} \rho^{-j}  \sum_{l \in \mathcal{C}} D_{n-j} E^l_{n-j} +  \sum_{l \in \mathcal{C}}\sup_{i\geq N+1}\{D_i \}\sum_{i=N+1}^\infty \rho^{-i} \\&=
     \sum_{j=0}^{N} \rho^{-j}  \sum_{l \in \mathcal{C}} D_{n-j} E^l_{n-j} +  d \frac{\rho^{-N}} {\rho - 1} \sup_{i\geq N+1}\{D_i \}
\end{align*}
So, $\mathbb{P}$-almost surely,
\[
\limsup_{n}  \frac {Y^i_n} {\rho^n n^{k_i}} \leq W(i) \sum_{j=0}^{N}  \rho^{-j}  \sum_{l \in \mathcal{C}} \mathbb{P}_i(Z^l_j > 0) +  \sum_{l \in \mathcal{C}}\sup_{j\geq N+1}\{D_j \}\sum_{j=N+1}^\infty \rho^{-j}
\]
Letting $N \rightarrow \infty$ we get 
\[
\limsup_{n}  \frac {Y^i_n} {\rho^n n^{k_i}} \leq \zeta_iW(i)\hspace{0.5cm} \mathbb{P}-\text{a.s.}
\]
For the bound from below, we note that
\[
    \frac {Y^i_n} {\rho^n} \geq \sum_{j=0}^{N}\rho^{-j}  \sum_{l \in \mathcal{C}} D_{n-j} E^l_{n-j} 
\]
so taking $\liminf\limits_n$ and then letting $N \rightarrow \infty$ we get
\[
\liminf_{n}\frac {Y^i_n} {\rho^n n^{k_i}} \geq \zeta_iW(i)\hspace{0.5cm} \text{a.s.}
\]
concluding the proof of the lemma.
\end{proof}

Now define for $i \in \mathcal{C}$
\[
M^i_n = \max\{\xi_{v_k} \ : \ v \in \mathbb{T}_n, \ v_k \sim i, \ k\leq n\}
\]
We will show that 
\begin{align}\label{t3:maxdisp1}
    \mathbb{P}\left(M^I_n \leq a_n x\right) \xrightarrow[ n \rightarrow \infty]{} \mathbb{E}\left[ \exp \big\{ - \zeta W(I) x^{-r} \big\}\right]
\end{align}
and if $i \neq I$
\begin{align}\label{t3:maxdisp2}
    \mathbb{P}\left(M^i_n \leq a_n x\right) \xrightarrow[ n \rightarrow \infty]{} 1.
\end{align}
As a consequence, of course 
\begin{align}\label{t3:maxdisp3}
        \mathbb{P}\left(M_n \leq a_n x\right) \xrightarrow[ n \rightarrow \infty]{} \mathbb{E}\left[ \exp \big\{ - \zeta W(I) x^{-r} \big\}\right]
\end{align}
where $M_n = \max\limits_{i \in \mathcal{C}} \{M^i_n\}$.
We readily calculate
\[
\mathbb{P}\left(M^i_n \leq a_n x\right) = \mathbb{E} \left[ F_i(a_nx)^{Y^i_n}\right] =   \exp{\left\{\frac {Y^i_n} {\rho^n n^k} \rho^n n^k \log F_i(a_n x)\right\}} .
\]
 Now, $a_n$ was chosen so that $\rho^n n^k (1- F_I(a_n )) \xrightarrow[ n \rightarrow \infty]{}  1$, so we proceed to
\[
\rho^n n^k \log F_I(a_n x) \sim -\rho^n n^k (1- F_I(a_n x )) = -n^k\rho^n (1- F_I(a_n) ) \frac { (1- F_I(a_n x )) } {(1- F_I(a_n ))} \xrightarrow[ n \rightarrow \infty]{} -x^{-r}.
\]
Hence by Lemma \ref{t3:lemmatotal}, \[F_I(a_nx)^{Y^I_n} \xrightarrow[ n \rightarrow \infty]{a.s.}  \exp \big\{ -\zeta W(I) x^{-r} \big\}.\]
 Using the dominated convergence theorem, this proves \eqref{t3:maxdisp1}. Similarly, for $i\neq I$,
\begin{equation}
\begin{aligned}\label{t3:eq1}
   &\frac {Y^i_n} {\rho_i^n n^{k_i}} \frac {\rho_i^n n^{k_i}} {\rho^n n^{k}} \rho^n n^k \left( 1-  F_i(a_n x) \right)=  \frac {Y^i_n} {\rho_i^n n^{k_i}} \frac {\rho_i^n n^{k_i}} {\rho^n n^{k}} \rho^n n^k (1- F_I(a_n x) ) \frac { (1- F_i(a_n x )) } {(1- F_I(a_n x ))} 
\end{aligned}
\end{equation}
Note again that $\rho^n n^k (1- F_I(a_n )) \xrightarrow[ n \rightarrow \infty]{}  1$, and by Lemma \ref{t3:lemmatotal} $\rho_i^{-n}n^{-k_i} Y^i_n$ has a finite limit. To take care of the remaining terms, we note that
\begin{align}\label{t3:eq2}
 \frac {\rho_i^n n^{k_i}} {\rho^n n^{k}}\frac { (1- F_i(a_n x )) } {(1- F_I(a_n x ))} \sim \left(\frac {\rho_i} {\rho} \right)^n \left(\rho^{1-\frac { r_i} {r}  }\right)^n  h(n)
\end{align}
where (see Remark \ref{t3:remark}) \[
h(n) = \frac{ L_i\left( L^{\#}\left( \rho^{\frac n r} n^{\frac k r}\right)\rho^{\frac n r} n^{\frac k r}\right) L^{\#}\left( \rho^{\frac n r} n^{\frac k r}\right)^{r_i}n^{\frac k r r_i}} {L\left( L^{\#}\left( \rho^{\frac n r} n^{\frac k r}\right)\rho^{\frac n r} n^{\frac k r}\right) L^{\#}\left( \rho^{\frac n r} n^{\frac k r}\right)^{r}n^{ k }}.
\] 
Observe 
\[
\left(\frac {\rho_i} {\rho} \right)^n \left(\rho^{1-\frac { r_i} {r}  }\right)^n  =\rho_i^n \rho^{-n\frac {r_i} r}
\]
and since $\rho$ and $r$ were chosen to satisfy
\[
   \rho^{\frac 1 r} = \max{\{\rho_a^{\frac 1 {r_i}} \ : \ a \preceq i \}}
\]
we have for some $\varepsilon > 0$
\[
\rho^{\frac {1-\varepsilon} r} > \rho^{\frac 1 {r_i}}.
\]
Hence,
\[ \rho_i^n \rho^{-n\frac {r_i} r} = \left(\rho_i^{\frac 1 {r_i}}\right)^{n r_i} \rho^{-n\frac {r_i} r} \leq \left(\rho^{\frac {1- \varepsilon} {r}}\right)^{n r_i}\rho^{-n\frac {r_i} r} = \left(\rho^{\frac {r_i \varepsilon} r}\right)^{-n}. \] As $h(n)$ satisfies $\frac {h(n)} {\rho^{\delta n}} \rightarrow 0$ for any $\delta>0$, the right-hand side in \eqref{t3:eq2} goes to 0 as $n \rightarrow \infty$. Using the dominated convergence theorem again, \eqref{t3:maxdisp2} is proven.
\begin{remark}\label{t3:remark2}\emph{
As we see from the proof, we can in fact make even stronger statement than \eqref{t3:maxdisp2}. That is, for $i\neq I$ and small enough $\varepsilon >0$,
\begin{equation*}
        \mathbb{P}\left(M^i_n \leq \rho^{\varepsilon n}a_n x\right) \xrightarrow[ n \rightarrow \infty]{} 1.
\end{equation*}
}
\end{remark}

From here we proceed by induction. If $m=1$, the theorem reduces to Theorem \ref{t1}, so the base case is proven. Assume now that the theorem holds for processes with $m-1$ classes for $m>1$. Then we can write
\begin{align}
    R_n = \max(R^1_n,R^2_n)
\end{align}where 
\begin{align*}
&R^1_n = \max \{S_v \ : \ |v|=n, \ \sigma(v) =i, i \in \bigcup_{i\leq m-1} \mathcal{C}_i\}\\
&R^2_n = \max \{S_v \ : \ |v| = n,\ \sigma(v) =i, i \in  \mathcal{C}_m\}.
\end{align*}
Let $(\beta,J)$ be a pair of a class and a type attaining
\[\max{\{\rho_a^{\frac 1 {r_i}} \ : \ a \preceq i , a \leq m-1\}},\]
and denote $\gamma = \rho_\beta$ and $q = r_J$.
First consider the case when $\gamma^{\frac 1 q} = \rho^{\frac 1 r} > \rho_m^{\frac 1 {r(m)}}$, where $r(m) = \min\limits_{i \in \mathcal{C}_m} r_i$. By induction assumption
\begin{align}\label{t3:upper1}
\limsup_{n \rightarrow \infty}\mathbb{P}\left(\frac {R_n} {a_n} \leq x\right)  \leq \limsup_{n \rightarrow \infty} \mathbb{P}\left(\frac {R^1_n} {a_n} \leq x\right) = \E[e^{-\zeta W(I)x^{-q}}].
\end{align}
For the lower bound, consider a modified process $\tilde{S}_v = \sum_{k=1}^n \max \left( \xi_{v_i}, 0  \right)$ where the displacements are nonnegative. Define analogously
\begin{align*}
&\tilde{R}^1_n = \max \{\tilde{S}_v \ : \ |v|=n, \ \sigma(v) \in \bigcup_{i\leq m-1} \mathcal{C}_i\}\\
&\tilde{R}^2_n = \max \{\tilde{S}_v \ : \ |v| = n,\ \sigma(v) \in  \mathcal{C}_m\} \\
&\tilde{R}_n = \max(\tilde{R}^1_n,\tilde{R}^2_n).
\end{align*}
Note that $\tilde{S}$ satisfies the same assumptions we made on $S$, so all previous results hold. Now for a particle $v \in \mathbb{T}_n$, $\sigma(v) \in \mathcal{C}_m$, we have 
\begin{align}\label{t3:eqq1}
S_v = S_v - S_{v^*} + S_{v^*} \leq S_v - S_{v^*} + R^1_{k^*} \leq n\max_{i \in \mathcal{C}_m}\tilde{M}^i_n + \tilde{R}^1_n
\end{align}
where $v^*$ is the last ancestor of $v$ from the first $m-1$ classes, $k^* = |v^*|$ is its generation, and 
\[
\tilde{M}^i_n = \max\{\xi_{v_k} \ : \ v \in \mathbb{T}_n, \ v_k \sim i, i \in \mathcal{C}_m \ k\leq n\}.
\] 
Taking the supremum over $v$ in \eqref{t3:eqq1}, we have
\[
R^2_n \leq  n\max_{i \in \mathcal{C}_m}\tilde{M}^i_n + \tilde{R}^1_n,
\]
and trivially,
\[
R^1_n \leq  n\max_{i \in \mathcal{C}_m}\tilde{M}^i_n + \tilde{R}^1_n.
\]
Since $I \notin \mathcal{C}_m $, by Remark \ref{t3:remark2}, 
\[\frac{n\max_{i \in \mathcal{C}_m}\tilde{M}^i_n} {a_n} \xrightarrow[n \rightarrow \infty]{d} 0.\]  Hence, 
\begin{align*}
\liminf_{n \rightarrow \infty}\mathbb{P}\left(\frac {R_n} {a_n} \leq x\right) &\geq \liminf_{n \rightarrow \infty} \mathbb{P} \left( n\max_{i \in \mathcal{C}_m}\tilde{M}^i_n + \tilde{R}^1_n  \leq a_n x\right) \\&= \liminf_{n \rightarrow \infty}\mathbb{P}\left(\frac {\tilde{R}^1_n} {a_n} \leq x\right) =  \E[e^{-\zeta W(I)x^{-q}}]
\end{align*}
Together with \eqref{t3:upper1}, we conclude 
\[
\lim_{n\rightarrow \infty}\mathbb{P}\left(\frac {R_n} {a_n} \leq x\right) =  \E[e^{-\zeta W(I)x^{-q}}].
\]
Now consider the case when $\gamma^{\frac 1 q} < \rho^{\frac 1 r}$. Then by induction assumption we know that $\frac {R^1_n} {a^1_n}$ converges in distribution, where for any $\varepsilon >0$, $a^1_n < \gamma^{\frac n q (1+\varepsilon)}$ for sufficiently large $n$. Since $a_n >  \rho^{\frac n r (1-\varepsilon)}$ for any $\varepsilon >0$ and sufficiently large $n$, we have $\frac{a^1_n} {a_n} \rightarrow 0$, therefore $\frac {R^1_n} {a_n}$ converges in distribution to $0$. 
Hence, we only need to examine 
\begin{align}
\lim_{n\rightarrow \infty}\mathbb{P}\left(\frac {R^2_n} {a_n} \leq x\right).
\end{align}
We begin by showing that for any $\varepsilon > 0$,
  \begin{equation}\label{t3:upperbound}
            \limsup_{n \rightarrow \infty}  \mathbb{P} \left( R^2_n > a_n x\right)\leq 1- \E[e^{-\zeta W(I)x^{-q}}]
    \end{equation}
Take $v \in \mathbb{T}_n$, $\sigma(v) \in \mathcal{C}_m$. Then
\[
S_v = S_v - S_{v^*} + S_{v^*} \leq \sum_{k=k^*(v) + 1}^{n}\max\left\{\xi_{v_{k}},0\right\} + n\max\limits_{i \neq I} M^i_n.
\]
Where $v^*$ is the last ancestor of $v$ from the preceding $m-1$ classes and $k^* = |v^*|$. Note that $X_v := \sum_{k=k^* + 1}^{n}\max\left\{\xi_{v_{k}},0\right\}$ is a random sum of independent random variables with regularly varying tails, where the heaviest tail is of the order $r_I$ (since now $I \in \mathcal{C}_m$). Taking maximum over $v$, we have
\[
R^2_n \leq R^*_n + n\max\limits_{i \neq I}M^i_n
\]
where \[R^*_n = \max_{\sigma(v) \in \mathcal{C}_m}X_v.\] 
Again by Remark \ref{t3:remark2}, \[\frac{n\max\limits_{i \neq I}M^i_n} {a_n} \xrightarrow[n \rightarrow \infty]{d} 0.\]
 Hence, we have reduced \eqref{t3:upperbound} to showing 
   \begin{equation}\label{t3:upperbound2}
            \limsup_{n \rightarrow \infty}  \mathbb{P} \left( R^*_n > a_n x\right)\leq 1- \E[e^{-\zeta W(I)x^{-q}}].
    \end{equation}
 Let \[
 M^*_n = \max_{i \in \mathcal{C}_M} M^i_n.
 \]
 By \eqref{t3:maxdisp1} and \eqref{t3:maxdisp2}, 
 \begin{align}\label{t3:eq5}
 \lim_{n\rightarrow \infty} \mathbb{P}\left(M^*_n \geq a_n x\right) =\E[e^{-\zeta W(I)x^{-q}}].
 \end{align}
 From here we proceed similarly to the irreducible case. Since
 \[
 \limsup_{n \rightarrow \infty}  \mathbb{P} \left( R^*_n > a_n x\right) \leq  \lim_{n\rightarrow \infty} \mathbb{P}\left(M^*_n \geq a_n x\right) + \mathbb{P} \left( R^*_n > a_n x, M^*_n \leq a_n(x-\varepsilon)\right),
 \]
 by \eqref{t3:eq5} we only need to show 
\begin{align*} \lim_{n \rightarrow \infty} \mathbb{P} \left( R^*_n > a_n x, M^*_n \leq a_n(x-\varepsilon)\right) = 0
\end{align*}

Let $X_n(a_n x,\infty) = \#\left\{v \ : \ |v| =n , \ \sigma(v) \in \mathcal{C}_m, \ X_v >a_n x\right\}$. Then
\begin{align*}
    &\mathbb{P} \left( R^*_n > a_n x, M^*_n \leq a_n(x-\varepsilon)\right)\leq \mathbb{E} \left[X_n(a_n x, \infty) \boldone_{\{M^*_n \leq a_n(x- \varepsilon)\}} \right] \\
    &\leq \sum_{l=1}^{n} \sum_{\vec{l } \ : \ |\vec{l}|_1=l}A_{\vec{l}}\left(F_{\vec{l}}^{(a_n (x - \varepsilon))}(\infty) - F_{\vec{l}}^{(a_n (x - \varepsilon))} (a_n x)\right)
\end{align*}
where the inner sum goes over all vectors $\vec{l} = (l_1, \dots l_{|\mathcal{C}_m|})$ satisfying $|\vec{l}|_1 = \sum_{j=1}^{|\mathcal{C}_m|}l_j = l$,
$A_{\vec{l}}(n)$ is the expected number of class $\mathcal{C}_m$ particles in the $n$-th generation, that had $l_j$ ancestors of each respective type in $\mathcal{C}_m$, and \[F_{\vec{l}}^y(x) = F_{d-m+1, l_1}^y \ast F^y_{d-m,l_2} \ast \dots \ast F^y_{d,l_m} (x)\]
Since the step distributions in $F_{\vec{l}}$ now only involve distributions with tails not heavier than $r_I=r$, we can follow the argument from the proof of \eqref{t1:eq8}, which yields for $x, \varepsilon, \delta >0$, $s \in (0,r)$ and $C > 0$,
\[
\left(F_{\vec{l}}^{(a_n (x - \varepsilon))}(\infty) - F_{\vec{l}}^{(a_n (x - \varepsilon))} (a_n x)\right) \leq C \left( \frac {l C_s} {a_{n}^s(x - \varepsilon)^s} \right)^{\frac {x(1-\delta)} {(x - \varepsilon)} }
\]
and as a consequence, for some $C' > 0$ and $p'>1$,
\begin{align*}
\sum_{l=1}^{n} \sum_{\vec{l } \ : \ |\vec{l}|_1=l}&A_{\vec{l}}\left(F_{\vec{l}}^{(a_n (x - \varepsilon))}(\infty) - F_{\vec{l}}^{(a_n (x - \varepsilon))} (a_n x)\right) \\&\leq C' \rho^{-n  {p'} }\mathbb{E}[|Z_n(m)|]= \leq C' \rho^{-n  {p'} } \sum_{j=1}^m (M^nZ_0)(d-m+j)
\end{align*}
By Jordan decomposition, for some $k>0$ and all $j=1,\dots,m$, $\rho^{-n}n^{-k}(M^nZ_0)(d-m+j)$ converges to a finite limit, hence
\[\lim_{n\rightarrow \infty}\mathbb{P} \left( R^*_n > a_n x, M^*_n \leq a_n(x-\varepsilon)\right) = 0\]This concludes the proof of \eqref{t3:upperbound2}, and in result \eqref{t3:upperbound}. 
To finish the proof of the Theorem we are left to show
\[
\liminf_{n \rightarrow \infty}\mathbb{P} ( R_n > a_n x) \geq 1- \E[e^{-\zeta W(I)x^{-q}}].
\]
Since
\[  \mathbb{P} ( R_n > a_n x) \geq \mathbb{P}(M_n > a_n (x+\varepsilon)) - \mathbb{P} ( R_n \leq a_n x, M_n > a_n(x+\varepsilon)).
\]
and, by \eqref{t3:maxdisp3},
 \begin{align}\label{t3:eq10}
 \lim_{n\rightarrow \infty} \mathbb{P}\left(M_n \geq a_n x\right) =\E[e^{-\zeta W(I)x^{-q}}],
 \end{align}
 the proof is reduced to showing
    \begin{equation}\label{th2:lowerbound}
            \mathbb{P} ( R_n \leq a_n x, M_n \geq a_n(x+\varepsilon))\xrightarrow[ n \rightarrow \infty]{} 0 .
    \end{equation}
Here we note that the arguments used to show \eqref{upperbound} in the proof of Theorem \ref{t1} were based solely on the logarithmic bound on the lower tails of the distributions and the exponential growth of $a_n$. Since these conditions are still satisfied, the calculations can be repeated with no significant modifications, concluding the proof of the Theorem.
\end{proof}

\subsection{Displacements with semi-exponential tails}
Analogously to the irreducible case, throughout this section we assume that the displacements are independent and admit semi-exponential tails:\begin{equation}\label{t4:semiexp}
    \mathbb{P}\left(\xi^j \geq t\right) = a_j(t) \exp\{-L_j(t)t^{r_j}\},
\end{equation} where $L_j, a_j$ are slowly varying functions such that $\frac {L_j(t)} {t^{1-r_j}}$ are eventually nonincreasing, and $r_j \in (0,1)$. We also assume that the displacements have finite moments.
  Let \begin{align*}
   &r =  \min\{r_j \ | \ j \in \mathcal{C} \}, \\ 
   &\mathcal{B} = \{ j \in \mathcal{C} \ | \ r_j=r\}, \\ 
   &\mathcal{A} = \{ a \preceq j \ | \ j \in \mathcal{B}\}\end{align*}
so that $\mathcal{A}$ is the set of classes preceding types that attain $r$. Our result is as follows.

\begin{theorem}\label{t4}
         Let $L(x) = \min\{L_j(x) \ | \ j \in \mathcal{B} \}$, $\rho = \max\limits_{a \in \mathcal{A}} \rho(a)$, and choose $\psi(n)$ to be a positive function satisfying
\begin{equation}\label{t4:norm}
\frac {L(\psi(n))\psi(n)^{r}} n \rightarrow 1.
\end{equation}

Then
\[
\frac {R_n} {\psi(n)} \xrightarrow{a.s.}  (\log \rho)^{\frac 1 {r}}.
\]

    \end{theorem}
    \begin{remark}\label{t4:remark}
    \emph{As in the irreducible case, the existence of $\psi(n)$ satisfying \eqref{t4:norm} is guaranteed by the result of de Bruijn \cite{DeBruijn}. Indeed, if $K(x)$ is the de Bruijn conjugate of $x \mapsto L\left(x^{\frac 1 r}\right)$, then we can take$\psi(n)=K(n)^{\frac 1 r}n^{\frac 1 r}$. In particular, this implies that for any $\varepsilon > 0$, \[
    n^{\frac 1 r (1-\varepsilon)} \leq \psi(n) \leq n^{\frac 1 r (1+\varepsilon)}
    \]
    for large enough $n$.
    }
\end{remark}
The main difference between this Theorem and Theorem \ref{t2} is that $\rho$ is not necessarily the principal eigenvalue of $M$. This is because the limit behavior is driven by the heaviest tail and the asymptotic number of particles attaining it. In the irreducible case, all types share the same growth rate, but as seen in \cite{kesten2}, the growth rate of particles of any given type is also driven by the types preceding it. To illustrate the issue, consider the following heuristic argument: start with an irreducible process as class 1, and append to it another process as class 2, which follows class 1. Denote by $r_1$ and $r_2$ the heaviest tails that appear in classes 1 and 2, respectively. Then there are 3 cases to consider: 
\begin{enumerate}
    \item $r_2 < r_1$. Then the normalization factor $\psi(n)$ should adhere to class 2, as it has a heavier maximum tail. As the number of class 2 particles grows as $\max\{\rho(1),\rho(2)\}^n$, $\rho$ in Theorem \ref{t4} is, in fact, the principal eigenvalue of $M$. It is perhaps worth noting that the number of class 2 particles can even grow as $n^k\rho^n$ for some $k>0$ if $\rho_1=\rho_2$ (see \cite{kesten2} for details), but this is of no consequence here, since the bounds used in our proof work on exponential scale.
    \item $r_2 = r_1$. Since both classes attain the heaviest tail in the process, this is in essence the same as case 1. The only difference lies in $\psi(n)$, but it is limited to a slowly varying function.
    \item $r_2 > r_1$. Since the normalizing factor $\psi(n)$ grows as $n^{\frac 1 {r_1}}$, class 2 is essentially irrelevant. In this case $\rho=\rho(1)$ even if $\rho(2)$ is greater.

\end{enumerate}
This argument can be iterated by adding more classes. This should provide some intuition useful for understanding the rationale behind the structure of the proof. The key ingredient will be the following lemma.

\begin{lemma}\label{lemma3}
Denote the number of particles of class $a$ in the $n$-th generation as $Z_n(a)$. If $a \preceq j$, then, for any $\varepsilon > 0$, there exist $k \in \mathbb{N}$, $\delta >0$ and $\beta,\gamma \in (0,1)$, such that for all $n$ large enough 
     \begin{equation}\label{l3:1}
         \mathbb{P}\left(Z^j_n \leq \delta  Z_{n-k}(a)\right) \leq \beta^n.
     \end{equation}
       and
    \begin{equation}\label{l3:2}
        \mathbb{P}\left(\frac{Z_{n}^j} {\rho(a)^n}< (1-\varepsilon)^n\right) \leq \gamma^n
    \end{equation}

\end{lemma}
\begin{proof}[Proof of the Lemma \ref{lemma3}]
    We prove the lemma by induction. First observe that if $m=1$, the model reduces to the irreducible one considered in previous section and the statement of the lemma is an immediate consequence of Lemmas \ref{lemma} and \ref{lemma2}. Now assume that it holds for processes with $m$ classes and consider one with $m+1$ classes.
    Observe that the statement follows immediately from the induction assumption if $j \in \mathcal{C}_l$ for some $l\leq m$, as the last type does not contribute to the previous ones, hence we only consider the case when $j \in \mathcal{C}_{m+1}$. We start by showing \eqref{l3:1}. First note that it is trivially true if $a = m+1$, so assume $a \leq m$. Fix $i \in \mathcal{C}_a$ and $k \in \mathbb{N}$ such that $M^k(i,j) > 0$ (recall that $a \preceq j$ means that such $k$ exists for all $i \in \mathcal{C}_a$). Since \eqref{l3:1} holds for first $m$ classes, the problem can be reduced to showing existence of $\beta \in (0,1)$ satisfying
    \begin{equation}\label{l3:eq1}
       \mathbb{P}\left(Z^j_n < \delta  Z^i_{n-k}\right) \leq \beta^n.
    \end{equation}
   The proof is again similar to that of Lemma \ref{lemma2}. We denote by $Z_k^{i \rightarrow j}$ a random variable distributed as $Z_k^j$ under $\mathbb{P}_i$. Consequently, $q = \mathbb{P} \left( Z_k^{i \rightarrow j} = 0 \right) < 1$. Observe
    \begin{align*}
            \mathbb{P}\left(Z^j_n < \delta  Z^i_{n-k} \right) &= \mathbb{P} \left( \sum_{r \in \mathcal{C}} \sum_{l = 1}^{Z^r_{n-k}} Z_k^{i \rightarrow j} (l)< \delta  Z^i_{n-k} \right) \\&\leq  \mathbb{P} \left(  \sum_{l = 1}^{Z^i_{n-k}} Z_k^{i \rightarrow j} (l)< \delta  Z^i_{n-k} \right) = \mathbb{E} \left[\Phi\left(Z_{n-k}^i\right)\right]
    \end{align*}
where 
\[
\Phi(s) =  \mathbb{P} \left(  \sum_{l = 1}^{s} Z_k^{i \rightarrow j}(l) < \delta  s \right).
\]
and $\{Z_k^{i \rightarrow j} (l)\}_{l >0} $ are indpendent copies of $Z_k^{i \rightarrow j}$. 
If $q = 0$, the statement of the Lemma is trivially true with $\delta = 1$. Hence, we only consider the case where $q \in (0,1)$. Denote $K_n = \#\{l \leq s: Z_k^{i \rightarrow j}(l) > 0\}$. Then 
 \begin{align*}
     \Phi(s) = &\mathbb{P}\left(\sum_{l=1}^{s} Z_k^{i \rightarrow j}(l) < \delta s \right) = \sum_{t=0}^\infty \mathbb{P}\left(\sum_{l=1}^{s} Z_k^{i \rightarrow j}(l) < \delta s \ , \ K_n =t\right) \\&\leq \sum_{t=0}^\infty \mathbb{P}(t <\delta s,K_n=t) = \sum_{t=0}^{\lfloor \delta s \rfloor} \binom{s}{t} (1- q)^t q^{s-t} \leq q^s\sum_{t=0}^{\lfloor \delta s \rfloor}\frac {(\delta s)^t} {t!} \left(\frac {1-q} {\delta q}\right)^t
 \end{align*}
 Choosing $\delta > 0$ small enough so that \[
 \frac {1-q} {\delta q} > 1,\] and \[
 \beta_0 = q \left(\frac {1-q} {\delta q}\right)^{\delta} < 1.
 \]we get the bound 
 \[
 \Phi(s) \leq \beta_0^s.
 \]
Then 
     \begin{align*}
    \mathbb{E} \left[\Phi\left(Z_{n-k}^i\right)\right] \leq \mathbb{E}\left[\beta_0^{Z_{n-k}^i} \right] \leq \beta_0^n + \mathbb{P}\left(Z_{n-k}^i< n\right).
    \end{align*}
    Since $i \in \mathcal{C}_a$ for $a \leq m$, $\mathbb{P}\left(Z_{n-k}^i< n\right)$ decays exponentially fast by induction assumption, which ends the proof of \eqref{l3:1}. Having proved that, \eqref{l3:2} follows easily:
    \begin{align*}
        \mathbb{P}\left(\frac {Z_{n}^j} {\rho(a)}< (1-\varepsilon)^n\right) \leq \beta^n + \mathbb{P}\left(\frac {\delta Z_{n}^i} {\rho(a)}< (1-\varepsilon)^n\right)  \leq \beta^n + \gamma_0^{\frac n {\delta}} \leq \gamma^n
    \end{align*}
    for appropriate choice of $\gamma \in (0,1)$ and large enough $n$.
    \end{proof}

\begin{proof}[Proof of the Theorem \ref{t4}]

We start with the upper bound. For this part, we can assume without loss of generality (replacing $\xi^i$ with $\max\{\xi^i,0\}$) that all $\xi^i$'s are nonnegative. Denote \[\theta = \max \{\alpha \ | \ \mathcal{C}_\alpha \subset \mathcal{B}\}.\] Note that we may assume that if $\alpha$ is not comparable (with respect to relation $\preceq$) with any class contained in $\mathcal{B}$, then $\alpha > \theta$. This is achieved by simple renumbering of some classes and is consistent with the ordering we assumed in \eqref{matrix}. Let \[R^\theta_n = \max \{ S_v \ | \ |v|=n, \ \sigma(v) \in \mathcal{C}_\alpha, \ \alpha \leq \theta\}\]

Denote by $M_\theta$ the minor of $M$ that includes only the subset of classes $\mathcal{C}_{\sim}^\theta = \{\mathcal{C}_1, \mathcal{C}_2, \dots, \mathcal{C}_\theta\}$, so
\begin{equation}\label{matrix2}
    M_\theta = \begin{pmatrix}
        M[1] & M[1,2] & \dots & M[1,\theta] \\
        0 & M[2] & \dots & M[2,\theta]\\
        \vdots & \vdots & \vdots & \cdots \\
        0 & \dots & 0 & M[\theta]
    \end{pmatrix}
    .
\end{equation}
Since the subsequent classes do not contribute to the previous ones, the sub-process consisting only of particles of these types is a multi-type Galton-Watson process with mean matrix $M_\theta$. 
We now repeat the construction from the first part of the proof of Theorem \ref{t2}. Recall that we bounded the tails of $S_v$ by a tail of an i.i.d. sum of random variables with the heaviest tails. When applied to the considered subprocess, this yields the following inequality. For any $\varepsilon > 0$, there is $\delta >0 $ such that for all large enough $n$,
    \begin{equation}
        \mathbb{P}\left(S_v \geq \psi(n) (\log \rho + \varepsilon)^{\frac 1 r}\right) \leq \exp \left\{ -n(\log \rho + \delta) \right\} = \rho^{-n} e^{ -n \delta}
    \end{equation} 
    for any $v \in \mathbb{T}^\theta_n = \{v \in \mathbb{T}_n \ | \ \sigma(v)  \in \mathcal{C}_\alpha, \alpha \leq \theta \ \}$.  
Calculation analogous to \eqref{calc} yields the inequality
\begin{equation}\label{calc2}
\begin{aligned}
    &\mathbb{P}\left(\exists v \in \mathbb{T}_n : S_v \geq \psi(n) (\log \rho + \varepsilon)^{\frac 1 r}\right) \leq 
    \mathbb{E}[|Z^\theta_n|] \rho^{-n} e^{ -n \delta},
\end{aligned}
\end{equation}
where $|Z^\theta_n|$ is the number of particles in our sub-process. Now observe that
\[
\mathbb{E}[|Z^\theta_n|] = \mathbb{E}\left[\left|M^n_\theta\left(Z_0\right)\right|\right] = \left|M^n_\theta\left(Z_0\right)\right|.
\]
As $\rho$ is the largest eigenvalue of $M_\theta$, by bringing $M_\theta$ to its Jordan form we see that for some $k\leq d$, $n^{-k}\rho^{-n} M_\theta^n$ has a finite limit, hence 
\[\mathbb{E}[|Z^\theta_n|] \rho^{-n} e^{ -n \delta} \leq C e^{-n\delta'}\]
for some $C>0, \delta' >0$, which by the Borel-Cantelli lemma proves
\begin{equation}\label{t4:lim}
    \limsup_{n \rightarrow \infty} \frac {R^\theta_n} {\psi(n)} \leq \left(\log \rho\right)^{\frac 1 r}
\end{equation}

Now we consider
\begin{equation}\label{matrix3}
    M_c = \begin{pmatrix}
        M[\theta + 1] & M[\theta+1,\theta+2] & \dots & M[\theta+1,m] \\
        0 & M[\theta+2] & \dots & M[\theta+2,m]\\
        \vdots & \vdots & \vdots & \cdots \\
        0 & \dots & 0 & M[m]
    \end{pmatrix}
\end{equation}
Observe that if $v \in \mathbb{T}_n / \mathbb{T}^\theta_n$, then it necessarily has the last ancestor from $\mathbb{T}^\theta_{k_v}$ (as the initial particle comes from the first class) for some $k_v$. We denote this ancestor $v^\theta$. Hence
\begin{equation} \label{t4:eq1}
    \frac{S_v} {\psi(n) }  = \frac{S_v - S_{v^\theta} + S_{v^\theta}} {\psi(n)}   \leq \frac{\sum_{i = k_v +1}^{n} \xi_{v_i} + R^\theta_{k_v}} {\psi(n)} \leq \frac{X_v}{\psi(n)} + \frac{R^\theta_{n}} {\psi(n)}
\end{equation}
where $\sum_{i = k_v +1}^{n} \xi_{v_i}$. Taking maximum over $v \in \mathbb{T}_n / \mathbb{T}^\theta_n$ and letting $n\rightarrow \infty$ in \eqref{t4:eq1}, we see that with \eqref{t4:lim}, we only need to show 
\begin{align}\label{t4:eq2}
\frac{R^*_n}{\psi(n)}\xrightarrow[n\rightarrow \infty]{a.s.} 0,\end{align} 
where
\[R^*_n = \max_{v \in \mathbb{T}_n / \mathbb{T}^\theta_n} X_v.\] 
From here we again deploy the strategy used in the proof of \eqref{calc}.  Let 
\begin{align*}&r' = \min \{r_i \ | \ i \in \bigcup_{\alpha > \theta} \mathcal{C}_\alpha \} > r \\
&L'(x) = \min \{ L_i(x) \ | \ i \in \bigcup_{\alpha > \theta} \mathcal{C}_\alpha, r_i = r'\} \\&\rho' = \max \{ \rho(\alpha) \ | \ \alpha > \theta\}.
\end{align*} furthermore, let $\Psi$ be a function satisfying
\begin{equation}\label{norm3}
\frac {L'(\Psi(n))\Psi(n)^{r'}} n \rightarrow 1.
\end{equation} Analogously to \eqref{calc}, we bound
\[\mathbb{P}\left( X_v > x\right)\leq \mathbb{P}\left(\tilde{X}_n > x\right)
\]
where 
\[
\tilde{X}_n = \sum_{i=1}^n \tilde{\xi}_i
\]
is a sum of independent nonnegative random variables with cumulative distribution function $F$, satisfying
    \[
      a_1(x)\exp\{-L'(x)x^{r'}\} \leq1-F(x) \leq a_2(x)\exp\{-L'(x)x^{r'}\} 
    \] 
    for some slowly varying $a_1, a_2$. As a consequence, by Theorem 3 from \cite{Gantert2000}, for any $\varepsilon > 0$, there is $\delta >0 $ such that for all large enough $n$,
    \begin{equation}
        \mathbb{P}\left(X_v \geq \{\Psi(n) (\log \rho' + \varepsilon)^{\frac 1 {r'}}\right) \leq \exp \left\{ -n(\log \rho' + \delta) \right\} = (\rho')^{-n} e^{ -n \delta}.
    \end{equation}
    As a consequence,
    \[ \mathbb{P}\left(\exists v \in \mathbb{T}_n : X_v \geq \{\Psi(n) (\log \rho' + \varepsilon)^{\frac 1 {r'}}\right) \leq \mathbb{E}\left[\sum_{\alpha > \theta} Z_n(\alpha)\right] (\rho')^{-n} e^{ -n \delta}.
    \]
Since  \[\mathbb{E}\left[\sum_{\alpha > \theta} Z_n(\alpha)\right] = \sum_{i \in \bigcup_{\alpha > \theta} \mathcal{C}_\alpha} M_n(i) \]
Again by Jordan decomposition, for all \[i \in \bigcup_{\alpha > \theta} \mathcal{C}_\alpha,\]
and some $k>0$,
$n^{-k} (\rho')^{-n}M_n(i)$ converges to a finite limit. We conclude
\[
\limsup_{n \rightarrow \infty} \frac{R*_n}{\Psi(n)} \leq \left(\log \rho'\right)^{\frac 1 {r'}}\ \ \text{ a.s.}
\]
By Remark \ref{t4:remark}, since $r<r'$, we have
\[
\frac {\Psi(n)} {\psi(n)} \xrightarrow[n\rightarrow \infty]{} 0
\]
therefore
\begin{equation}\label{t4:eq3}
    \limsup_{n \rightarrow \infty} \frac {R^*_n} {\psi(n)} \leq 0\ \ \text{ a.s.}
\end{equation}
Thus we have proved
\[
    \limsup_{n \rightarrow \infty} \frac {R_n} {\psi(n)} \leq \left(\log \rho\right)^{\frac 1 {r}}\ \ \text{ a.s.}
\]

For the lower bound we again apply the trimming procedure used in the proof of Theorem \ref{t2} and reduce the problem to the convergence of the series 
    \begin{align*}\label{sum}
        \sum_{n=1}^\infty \mathbb{P}\left(\frac {Z^{I}_n} {\rho^n} \leq \left(1-\frac {\varepsilon_1} 2\right)^n \right) 
    \end{align*}
where $I \in \mathcal{B}$. Recall that $\rho$ was the maximum eigenvalue among the classes followed by the types in $\mathcal{B}$, so applying Lemma \ref{lemma3} (more specifically \eqref{l3:2}) ends the proof.

\end{proof}

\end{document}